\documentclass{amsart}

\usepackage{amssymb}
\usepackage{mathrsfs}

\usepackage{hyperref}

\usepackage{booktabs,caption}
\usepackage{longtable}
\usepackage{enumitem}

\usepackage{comment}
\usepackage{graphicx}
\usepackage{color}
\usepackage[all]{xy}
\usepackage[initials]{amsrefs}
\newtheorem{theorem}{Theorem}[section]
\newtheorem{lemma}[theorem]{Lemma}
\newtheorem{proposition}[theorem]{Proposition}

\theoremstyle{definition}

\newtheorem{remark}[theorem]{Remark}

\renewcommand{\setminus}{\smallsetminus}
\renewcommand{\emptyset}{\varnothing}

\newcommand{\into}{\hookrightarrow} %monomorfismo
\newcommand{\onto}{\twoheadrightarrow} %epimorfismo
\newcommand{\id}{\mathrm{id}} %identita`
\newcommand{\Ext}{\mathrm{Ext}}
\newcommand{\cExt}{\mathcal{E}xt}

\DeclareMathOperator{\Hom}{Hom} %spazio degli omomorfismi
\DeclareMathOperator{\Spec}{Spec} %spettro di un anello
\DeclareMathOperator{\Proj}{Proj} %spazio proiettivo
\DeclareMathOperator{\coker}{coker}
\DeclareMathOperator{\Pic}{Pic}
\DeclareMathOperator{\rank}{rank}

\newcommand{\Div}{\mathrm{Div}}

\def\conv#1{\mathrm{conv} \left\{ #1  \right\} } %convex hull
\def\cone#1{\mathrm{cone} \left\{ #1 \right\}}

\def\pow#1{[ \! [ #1 ] \! ] }
\newcommand\cO{\mathcal{O}}
\newcommand\cR{\mathcal{R}}
\newcommand\cX{\mathscr{X}}

\renewcommand\AA{\mathbb{A}}
\newcommand\CC{\mathbb{C}}
\newcommand\FF{\mathbb{F}}
\newcommand\GG{\mathbb{G}}
\newcommand\NN{\mathbb{N}}
\newcommand\PP{\mathbb{P}}
\newcommand\QQ{\mathbb{Q}}
\newcommand\RR{\mathbb{R}}
\newcommand\ZZ{\mathbb{Z}}

\newcommand\rH{\mathrm{H}}

\newcommand\rR{\mathrm{R}}

\newcommand\rmm{\mathrm{m}}

\newcommand{\morimukai}[2]{$\mathrm{MM}_{#1 - #2}$} %mori-mukai list

\newcommand{\stack}[2]{ \mathcal{M}^\mathrm{Kss}_{#1, #2} } %K-moduli stack
\newcommand{\modspace}[2]{\mathrm{M}^\mathrm{Kps}_{#1, #2}} %K-moduli space

\newcommand{\coh}[2]{\mathrm{H}^{#1}(#2, \mathbb{Q})} %cohomology

\newcommand\cT{\mathscr{T}}
\newcommand\TT{\mathbb{T}}
\newcommand\cZ{\mathscr{Z}}
\DeclareMathOperator\Def{Def}
\DeclareMathOperator\DefqG{Def^{\mathrm{qG}}}
\DeclareMathOperator\Aut{Aut}
\newcommand{\GL}{\mathrm{GL}}

\newcommand{\TTqG}[2]{\mathbb{T}^{\mathrm{qG}, #1}_{#2}}
\newcommand{\cTqG}[2]{\mathscr{T}^{\mathrm{qG}, #1}_{#2}}

\newcommand{\epsi}{\varepsilon}
\newcommand{\frakX}{\mathfrak{X}}

\newcommand{\LL}{\mathbb{L}}

\DeclareMathOperator{\Sing}{Sing}
\DeclareMathOperator{\spe}{sp}

\title[On toric geometry and K-stability of Fano varieties]{On toric geometry and K-stability \\ of Fano varieties}

\author{Anne-Sophie Kaloghiros}

\address{Department of Mathematics, Brunel University London, Kingston Lane, Uxbridge UB8 3PH, United Kingdom}

\email{anne-sophie.kaloghiros@brunel.ac.uk}

\author{Andrea Petracci}

\address{Institut f\"ur Mathematik, Freie Universit\"at Berlin, Arnimallee 3, Berlin 14195, Germany}

\email{andrea.petracci@fu-berlin.de}

\begin{document}
   \begin{abstract}
We present some applications of the deformation theory of toric Fano varieties to K-(semi/poly)stability of Fano varieties.
First, we present two examples of K-polystable toric Fano $3$-fold with obstructed deformations. In one case, the K-moduli spaces and stacks are reducible near the closed point associated to the toric Fano $3$-fold, while in the other they are non-reduced near the closed point associated to the toric Fano $3$-fold.
Second, we study K-stability of the general members of two deformation families of smooth Fano $3$-folds by building degenerations to K-polystable toric Fano $3$-folds.
%Second, we use openness of K-semistability to show that the general members of two deformation families of smooth Fano $3$-folds are K-semistable by building degenerations to K-polystable toric Fano $3$-folds. 
\end{abstract}

   \maketitle
   \section{Introduction}
   \label{sec:Intro}
 
In this paper, we present some applications of the deformation theory of toric Fano varieties to K-(semi/poly)stability of Fano varieties.
Working with toric varieties enables us to run many computations explicitly, and to analyse the local structure of some K-moduli spaces and stacks of $3$-dimensional Fano varieties.   
%The great advantage of working with toric varieties is that many computations can be run completely explicitly.

In the first part of the paper, we study two examples of K-polystable toric Fano $3$-folds with obstructed deformations.
These define non-smooth points of both the K-moduli stack of K-semistable Fano $3$-folds and of the K-moduli space of K-polystable Fano $3$-folds. In one case, the K-moduli stack has $4$ branches and the K-moduli space has $3$ branches. In the other case, the K-moduli space is a fat point. These are, to the best of our knowledge, the first examples of such behaviour. 

Second, we establish the K-polystability of the general member of two families of smooth Fano $3$-folds by showing that they arise as smoothings of K-polystable toric Fano $3$-folds explicitly.
In one of these cases, K-semistability was not known, while in the other, our argument provides an alternative proof. 

We now state the main results of the paper and present its organisation.

\subsection{Non-smooth K-moduli}

An immediate consequence of Kodaira--Nakano vanishing is that deformations of smooth Fano varieties are unobstructed.
It follows that moduli stacks of smooth Fano varieties are smooth.
It is also known that $\QQ$-Gorenstein deformations (i.e.\ those satisfying Koll\'ar's condition) of del Pezzo surfaces with cyclic quotient singularities are unobstructed \cite[Proposition~3.1]{hacking_prokhorov} \cite[Lemma~6]{procams}. As in the smooth case, this implies that moduli of del Pezzo surfaces with cyclic quotient singularities are smooth \cite{odaka2016} (see also Proposition~\ref{prop:unobstructed_del_Pezzo}).

In dimension $3$, there are examples of Fano varieties with obstructed deformations and isolated (canonical) singularities \cite{jahnke_radloff, petracci_roma, petracci_survey}. Note however that Fano $3$-folds with terminal singularities have unobstructed deformations \cite{namikawa_fano, sano_fano}. 

In light of recent developments in the moduli theory of Fano varieties, it is natural to ask whether deformations of K-semistable or of K-polystable Fano $3$-folds are obstructed. 
For example, \cite{liu_xu_cubic} have shown that  the K-moduli stack of K-semistable cubic $3$-folds coincides with the GIT stack, and is therefore smooth.
% Perhaps add a sentence about singularities of GIT quotients?

We show that the naive hope that deformations of K-polystable Fano $3$-folds would be unobstructed is not validated. 

In what follows, for $n\geq 1$ and $V\in \QQ$, $\stack{n}{V}$ denotes the moduli stack of $\QQ$-Gorenstein families of K-semistable Fano varieties of dimension $n$ with anticanonical degree $V$ and $\modspace{n}{V}$ denotes its good moduli space (see \S\ref{sec:KstabFa}).

\begin{theorem}
There exists a K-polystable toric Fano $3$-fold $X$ with Gorenstein canonical singularities and anticanonical volume $12$ such that:
\begin{enumerate}
    \item the stack $\stack{3}{12}$ and the algebraic space $\modspace{3}{12}$ are not smooth at the point corresponding to $X$;
    \item for every $n \geq 4$, if $V = 2n(n-1)(n-2)^{n-2}$ then the stack $\stack{n}{V}$ and the algebraic space $\modspace{n}{V}$ are not smooth at the point corresponding to $X \times \PP^{n-3}$.
\end{enumerate}
\end{theorem}

The local structure of $\stack{3}{12}$ and of $\modspace{3}{12}$ near $[X]$ are studied in detail in \S\ref{Kmodsp} (see Theorem~\ref{Kmod-ex}); we show that the base of the miniversal deformation (Kuranishi family) of $X$ has $4$ irreducible components, and:
\begin{enumerate}[label=$\bullet$]
    \item one component parametrises deformations of $X$ to a smooth Fano $3$-fold in the deformation family \morimukai{2}{6} (Picard rank $2$, degree $12$ and $h^{1,2} = 9$);
    \item  one component parametrises deformations of $X$ to a smooth Fano $3$-fold in the deformation family \morimukai{3}{1} (Picard rank $3$, degree $12$);
    \item the remaining two components parametrise deformations of $X$ to smooth Fano $3$-folds in the deformation family $V_{12}$ (Picard rank $1$, degree $12$).
\end{enumerate}
This implies that $\stack{3}{12}$ has $4$ branches at $[X]$. Two of these branches are identified when passing to the good moduli space, hence $\modspace{3}{12}$ has $3$ branches at $[X]$.

\medskip

A second example of a K-polystable toric Fano $3$-fold with canonical singularities and obstructed deformations gives:

\begin{theorem}\label{non-red}
There is a connected component of the K-moduli space $\modspace{3}{44/3}$ isomorphic to  $\Spec \Big( \CC[t]/(t^2)\Big)$.\end{theorem}

In general,  this shows that K-moduli stacks and spaces can be both reducible and non-reduced.

\subsection{K-(semi/ poly)stability of smooth Fano $3$-folds by degenerations to toric varieties} 

Recent works have shown that K-semistability is an open property \cite{BLX, xu_minimizing} (see \S\ref{sec:KstabFa}).
In particular, if a smooth Fano $3$-fold is a general fibre in a $\QQ$-Gorenstein smoothing of a K-polystable Fano $3$-fold, then it is automatically K-semistable.   
In \S\ref{sec:Ksstab}, we construct $\QQ$-Gorenstein smoothings of two K-polystable toric Fano $3$-folds, and conclude that the general member of the deformation of each smoothing is K-semistable.  Using the local structure of K-moduli described in \cite{luna_etale_slice_stacks}, we show:
\begin{theorem}\label{refimproved}
The general member of the deformation family of \morimukai{2}{10} (Picard rank $2$, degree $16$, $h^{1,2}=3$) is K-stable. The general member of the deformation family of \morimukai{4}{3} (Picard rank $4$, degree $28$, $h^{1,2} = 1$) is K-polystable. 
\end{theorem}
\begin{remark}
 The general member of the deformation family \morimukai{2}{10} is known to be K-polystable by \cite[Section 4.3]{ACC+}, where an example of K-semistable \morimukai{2}{10} with symmetries is constructed. The K-semistability of a general member of family \morimukai{4}{3} was not previously known. It has now been proved that every smooth member of that deformation family is K-polystable \cite[Section 4.6]{ACC+}.
\end{remark}

   \subsection{Notation and conventions}
All varieties, schemes and stacks considered in this paper are defined over $\CC$.
A normal projective variety is \emph{Fano} if its anticanonical divisor is $\QQ$-Cartier and ample.
A \emph{del Pezzo surface} is a $2$-dimensional Fano variety.
We only consider normal toric varieties. 

The symbol $V_k$ denotes the deformation family of smooth Fano $3$-folds of Picard rank $1$, Fano index $1$ and degree $k$.
The symbol \morimukai{\rho}{k} denotes the $k$th entry in the Mori--Mukai list \cite{MM86,MM03} of smooth
Fano $3$-folds of Picard rank $\rho$, with the exception of the case $\rho = 4$,
where we place the $13$th entry in Mori and Mukai's rank-4 list in between
the first and the second elements of that list.
This reordering ensures that, for each $\rho$, the sequence
     \morimukai{\rho}{1}, \morimukai{\rho}{2}, \morimukai{\rho}{3}, ...
is in order of increasing degree.

If $Z$ is a normal scheme of finite type over $\CC$, then $\Omega^1_Z$ denotes the sheaf of K\"ahler differentials of $Z$ over $\CC$. For $i \in \{0,1,2\}$, we write  $\TT^i_Z$ for the $\CC$-vector space
$\Ext^i (\Omega^1_Z, \cO_Z)$, and $\cT^i_Z$ for the coherent $\cO_Z$-module $\cExt^i (\Omega^1_Z, \cO_Z)$.
The $\QQ$-Gorenstein versions of these are described in \S\ref{sec:QGorenstein}.

If $M$ is a topological space and $i \geq 0$ is an integer, then $\coh{i}{M}$ denotes the $i$th singular cohomology group of $M$ with coefficients in $\QQ$ and $b_i(M)$ denotes the $i$th Betti number of $M$, i.e.\ the dimension of $\coh{i}{M}$. The topological Euler--Poincar\'e characteristic of  a topological space $M$ is denoted by $\chi(M)$.

If $Z$ is a scheme of finite type over $\CC$, when considering topological properties we always consider the analytic topology on the set of the $\CC$-points of $Z$.

%\begin{comment}
\subsection*{Acknowledgements}
ASK wishes to thank Ivan Cheltsov and Kento Fujita for several useful conversations on K-stability of $3$-folds.
AP wishes to thank Alessio Corti and Paul Hacking for many conversations during the preparation of \cite{corti_hacking_petracci}, which have been very helpful for this paper; he is grateful to Hendrik S\"u\ss{} for asking him how deformations of toric Fanos are related to K-stability and to Taro Sano for pointing \cite{christophersen_ilten} to him.
Both authors wish to thank Yuji Odaka for fruitful conversations, and the referee for their comments, and in particular for suggesting the stronger statement of Theorem~\ref{refimproved}. 
%\end{comment}

   \section{Preliminaries}
   \label{sec:Preliminaries}
   
   In this section, we collect some results that will be used throughout the paper. 

\subsection{$\QQ$-Gorenstein deformations} \label{sec:QGorenstein}

An important insight, originally due to Koll\'ar and Shepherd-Barron \cite{kollar_shepherd_barron}, is that one should only consider $\QQ$-Gorenstein deformations when studying the moduli theory of higher dimensional singular varieties.
Roughly speaking, $\QQ$-Gorenstein families are flat families for which the canonical classes of fibres fit together well.  More formally, if $X$ is $\QQ$-Gorenstein, then a $\QQ$-Gorenstein deformation is one that is induced by a deformation of the canonical cover stack of $X$ (see \cite{abramovich_hassett_stable_varieties, moduli_products}).
The canonical cover $\frakX$ is a Deligne--Mumford stack with coarse moduli space $X$ and such that $\frakX\to X$ is an isomorphism over the Gorenstein locus of $X$. If $X$ is Gorenstein, then $\frakX\simeq X$ and any deformation is a $\QQ$-Gorenstein deformation. 
We denote by $\TTqG{i}{X}$ (resp.\ $\cTqG{i}{X}$) the $i$th Ext group (resp.\ the pushforward to $X$ of the Ext sheaf) of the cotangent complex of $\frakX$.
There is a spectral sequence
\[
E^{p,q}_2 = \rH^p (X, \cTqG{q}{X} ) \Longrightarrow \TTqG{p+q}{X}.
\]
As usual, $\TTqG{1}{X}$ is the tangent space and $\TTqG{2}{X}$ is an obstruction space for the $\QQ$-Gorenstein deformation functor $\DefqG X$ of $X$.

\subsection{K-stability of Fano varieties}\label{sec:KstabFa}
 The notion of K-stability was introduced by Tian \cite{tian_KE} in an attempt to characterise the existence of K\"ahler--Einstein metrics on Fano manifolds; it was later reformulated in purely algebraic terms in \cite{donaldson_stability}. We will not define the notions of K-semistability, K-polystability or K-stability here; we refer the reader to the survey \cite{xu_survey} and to the references therein. The notion of K-stability has received significant interest from algebraic geometers in recent years, as it has become clear that it provides the right framework to construct well behaved moduli stacks and spaces for Fano varieties \cite{odaka2016,ABHLX, BLX,li2019}. 

 Many proofs of K-semistability of Fano varieties rely on the following:
 
\begin{theorem}[{\cite{BLX, xu_minimizing}}]
    \label{openness}
    If $\mathcal{X}\to B$ is a $\QQ$-Gorenstein family of Fano varieties, then the locus where the fibre is a K-semistable variety is an open set. 
\end{theorem}

In particular if a K-polystable $\QQ$-Gorenstein toric Fano $3$-fold admits a $\QQ$-Gorenstein smoothing to a Fano $3$-fold in a given deformation family, then the general member of that deformation family is K-semistable. 

For every integer $n \geq 1$ and every rational number $V > 0$, 
let $\stack{n}{V}$ denote the category fibred in groupoids over the category
 of $\CC$-schemes defined as follows: for every $\CC$-scheme $B$, $\stack{n}{V}(B)$ is the  groupoid of $\QQ$-Gorenstein flat proper finitely presented families with base $B$ of 
 K-semistable klt Fano varieties of dimension $n$ and anticanonical volume $V$.

\begin{theorem}[{\cite{ABHLX, xu_minimizing, BLX, jiang_boundedness, blum_xu_uniqueness}}]
$\stack{n}{V}$ is an Artin stack of finite type over $\CC$ and admits a good moduli space $\modspace{n}{V}$, which is a separated algebraic space of finite type over $\CC$.
Moreover, $\modspace{n}{V}(\CC)$ is the set of K-polystable Fano $n$-folds with anticanonical volume $V$.
\end{theorem}

The notion of good moduli space is defined in \cite{alper_good_moduli_spaces}. The stack $\stack{n}{V}$ is called the K-moduli stack, and the algebraic space $\modspace{n}{V}$ is called the K-moduli space.

\medskip

Now we describe the local structure of K-moduli explicitly.
Let $X$ be a K-polystable Fano variety of dimension $n$ and anticanonical degree $V$.
Let $A$ be the noetherian complete local $\CC$-algebra with residue field $\CC$ which is the hull of the functor of $\QQ$-Gorenstein deformations of $X$, i.e.\ the formal spectrum of $A$ is the base of the miniversal $\QQ$-Gorenstein deformation of $X$.
Let $G$ be the automorphism group of $X$; then $G$ is reductive by \cite{ABHLX}.
The group $G$ acts on $A$ and the Luna \'etale slice theorem for algebraic stacks \cite{luna_etale_slice_stacks} gives in this case a cartesian square
\[
\xymatrix{
    [\Spec A \ / \ G] \ar[d] \ar[r] & \stack{n}{V} \ar[d] \\
    \Spec A^G \ar[r] & \modspace{n}{V}
}
\]
where the horizontal arrows are formally \'etale and map the closed point into the point corresponding to $X$.

In the case of del Pezzo surfaces, this description of K-moduli implies the following properties of the moduli stack and space. 

\begin{proposition} \label{prop:unobstructed_del_Pezzo}
For every $V \in \QQ_{>0}$, the stack $\stack{2}{V}$ is smooth and the moduli space $\modspace{2}{V}$ is normal, Cohen--Macaulay, with rational singularities.
\end{proposition}

\begin{proof}
Let $X$ be a K-polystable del Pezzo surface with $K^2_X = V$ and let $A$ be the hull of the functor of $\QQ$-Gorenstein deformations of $X$.
By \cite[Proposition~3.1]{hacking_prokhorov} or \cite[Lemma~6]{procams} $A$ is a power series $\CC$-algebra in finitely many variables; this implies that $\stack{2}{V}$ is smooth at $[X]$.

Let $G$ denote the automorphism group of $X$.
Since $A$ is normal, $A^G$ is normal.
From the reductivity of $G$ it follows that $A^G$ is Cohen--Macaulay by \cite{hochster_roberts}
 and that it has rational singularities by \cite{boutot}.
\end{proof}

\begin{remark} The proof of Proposition~\ref{prop:unobstructed_del_Pezzo} shows that $\stack{n}{V}$ is smooth and $\modspace{n}{V}$ is normal, Cohen Macaulay and has rational singularities in a neighbourhood of every K-polystable Fano variety with unobstructed $\QQ$-Gorenstein deformations (in particular, in the neighbourhood of every K-polystable smooth Fano variety).\end{remark}

   \subsection{Topology of smoothings}\label{sec:topo_smooth}
We recall the formalism of vanishing and nearby cycles and show how they relate the topology of the central fibre to that of a smoothing.  

Let $\cZ$ be a complex analytic space and $f \colon \cZ \to \Delta$ a flat projective morphism to a complex disc.
Denote by $Z$ the central fibre $ f^{-1}(0)$ and set $\cZ_t=  f^{-1}(t)$ for $t\neq 0$. There is a diagram of spaces and maps: 
\[
\xymatrix{
Z \ar@{^{(}->}[r]^i \ar[d]  & \cZ \ar[d]_f && TZ\smallsetminus Z \ar[d]  \ar@{_{(}->}[ll]_j  & \widetilde{TZ\smallsetminus Z} \ar[l]_{p} \ar[d] \\
\{0\} \ar@{^{(}->}[r]^{i_0} &\Delta && \Delta_e^*  \ar@{_{(}->}[ll]_{j_0} & \widetilde{ \Delta_e^*} \ar[l]_{p_0}
}
\]
where $\Delta_e= \{ z\in \CC : |z|<e\}$ is chosen so that $f$ restricts to a topologically trivial fibration $f^{-1}(\Delta^*_e) \to \Delta^*_e$, where $\Delta_e^* = \Delta_e \setminus \{0\}$. Let $TZ= f^{-1}\Delta_e$ be the tube about the fibre $Z$, $p_0\colon \widetilde{\Delta_e^*} \to \Delta_e^*$ the universal covering map, and let $p$ be the map making the right hand square of the diagram Cartesian. 
The nearby cycle complex of $f$ with $\QQ$-coefficients is:  
\begin{equation*} \psi_f \underline \QQ_{\cZ}^\bullet= i^*\Big(R(j\circ p)_*\circ (j\circ p)^*\Big)  \underline \QQ_{\cZ}^\bullet, \end{equation*}
where $\underline \QQ_{\cZ}^\bullet$ denotes the constant sheaf treated as a complex in degree $0$.  

As $f$ is projective, for any $t\in \Delta_e^*$, %proper enough
 there is a specialization map
\[  \spe \colon \cZ_t\stackrel{i_t}\hookrightarrow  f^{-1}(\Delta_e)\simeq Z\] and the vanishing cycle complex is \[\psi_f \underline \QQ_{\cZ}^\bullet \simeq R\spe_* \big( \underline \QQ_{\cZ_t}^\bullet\big).\]

%note that TZ and $\widetilde{TZ\smallsetminus Z}$ are analytic spaces

For any $x\in Z$, there is a natural isomorphism $\mathcal{H}^k(\psi_f \underline \QQ^\bullet_{\cZ})_x\simeq \coh{k}{F_x}$, where $F_x$ is the Milnor fibre of the function $f$ at $x$ \cite[Prop. 4.2.2]{Dimca}. Further, as $f$ is projective,  the hypercohomology of $\psi_f  \underline \QQ^\bullet_{\cZ}$ is, for all $k\in \ZZ$:
\[ \mathbb H^k(Z, \psi_f  \underline \QQ_{\cZ}^\bullet) \simeq 
\coh{k}{\cZ_t}.\]

The vanishing cycles complex of $f$ with $\QQ$-coefficients is obtained by considering the distinguished triangle 
\begin{equation}\label{coneq} i^* \underline \QQ_{\cZ}^\bullet \to \psi_f \underline \QQ_{\cZ}^\bullet\to \varphi_f \underline \QQ^\bullet_{\cZ} \stackrel{[1]}\rightarrow \end{equation}
associated to the cone of the specialisation map $i^* \underline \QQ_{\cZ}^\bullet\to \psi_f  \underline \QQ^\bullet_{\cZ}$. By definition, for all $x\in Z$
\[ \mathcal H^k(\varphi_f  \underline \QQ^\bullet_{\cZ})_x \simeq \widetilde{\rH}^k(F_x , \QQ).\]

The long exact sequence of hypercohomology associated to \eqref{coneq} is:
\begin{equation} \label{les}
\cdots \to \coh{k}{Z} \to \coh{k}{\cZ_t}  \to \mathbb{H}^k(Z, \varphi_f \underline \QQ^\bullet_{\cZ})\to \cdots\end{equation}
and since the fibres of $f$ are compact, we have:
\begin{equation} \label{topology_smoothing}
 \chi(\cZ_t) = \chi(Z)+ \chi(\mathbb H^\bullet(Z, \varphi_f \underline \QQ^\bullet_{\cZ})).\end{equation}

   \subsection{Toric varieties}\label{sec:tv} Toric varieties are a useful source of examples because they have an explicit combinatorial description; a comprehensive reference is \cite{cox_toric_varieties}. Here, we recall some of the terms and notions we will use. 
Given a lattice $N$ and a fan $\Sigma$ in $N$, one constructs the toric variety $X_\Sigma$, a normal variety endowed with an action of the torus $T_N = N \otimes_\ZZ \CC^* = \Spec \CC[M]$, where $M$ is the dual lattice $\Hom_\ZZ(N, \ZZ)$. 
The dimension of $X_\Sigma$ is the rank of $N$, and 
there is a $1$-to-$1$ correspondence between $i$-dimensional cones of $\Sigma$ and $(n-i)$-dimensional torus orbits in $X_\Sigma$, where $n = \dim X_\Sigma$.

Many geometric properties of $X_\Sigma$ can be read off from the fan $\Sigma$; for instance $X_\Sigma$ is complete if and only if $\Sigma$ covers all $N$. The next lemma shows how to read the Betti numbers of a complete toric $3$-fold $X_\Sigma$ off the fan $\Sigma$. 

\begin{proposition}[{Jordan \cite{jordan_phd}}] \label{prop:jordan}
    Let $\Sigma$ be a complete fan in a 3-dimensional lattice $N$ and let $X$ be the associated toric $3$-fold. For every $i=1,2,3$, let $d_i$ be the number of $i$-dimensional cones in $\Sigma$.
    Then: 
    \begin{gather*}
        d_1 - d_2 + d_3 = 2 \\
        b_0(X) = b_6(X) = 1, \\
        b_1(X) = b_5(X) = 0, \\
        b_2(X) = \rank \Pic(X), \\
        b_3(X) = \rank \Pic(X) - d_2 + 2 d_1 - 3, \\
        b_4(X) = d_1 - 3, \\
        \chi(X) = d_3.
    \end{gather*}
\end{proposition}

\begin{proof}[Sketch of proof]
    The positive dimensional cones of $\Sigma$ induce a polyhedral complex with support  $\bigcup_{\sigma \in \Sigma(3)} \conv{\rho \mid \rho \in \sigma(1)}$, which is homeomorphic to the 2-dimensional sphere. This implies the first equality.
    
    The equality $b_2(X)= \rank \Pic(X)$ follows from the fact that the first Chern class is an isomorphism from $\Pic(X)= \rH^1(X, \cO_X^*)$ to $\rH^2(X, \ZZ)$ \cite[Theorem~12.3.2]{cox_toric_varieties}.
    All other equalities come from the study of singular cohomology of toric varieties \cite[\S12.3]{cox_toric_varieties}; see \cite[Proposition~3.5.3]{jordan_phd}.
\end{proof}

We briefly recall the construction of polarised projective toric varieties \cite[Chapters 4 and 6]{cox_toric_varieties} (see also \cite[Lemma~2.3]{petracci_mavlyutov}).
Let $N$ be a lattice, $M$ its dual, and denote by $\langle \cdot, \cdot \rangle \colon M \times N \to \ZZ$ the duality pairing.
Let $Q$ be a full dimensional rational polytope in $M_\RR := M \otimes_\ZZ \RR$.
The \emph{normal fan} of $Q$ is the fan in $N$ whose maximal cones are
\[
\{x \in N_\RR \ \mid \ \forall q \in Q, \langle q -v, x \rangle \geq 0 \}
\]
where $v$ is a vertex of $Q$.
Denote by $\PP(Q)$ the toric variety associated to the normal fan of $Q$; $\PP(Q)$ is projective.
There is a $1$-to-$1$ correspondence between $i$-dimensional faces of $Q$ and $i$-dimensional torus orbits of $\PP(Q)$.

One may construct an ample $\QQ$-Cartier $\QQ$-divisor $D_Q\subset \PP(Q)$ supported on the torus-invariant prime divisors, with the property that for every integer $r \geq 0$, there is a canonical bijection between the lattice points of $rQ$ (the factor $r$ dilation of $Q$) and the monomial basis of $\rH^0(\PP(Q), r D_Q)$. 
Therefore
\[
\PP(Q) = \Proj \CC[ \cone{Q \times \{1\}} \cap (M \oplus \ZZ)]
\]
where the $\NN$-grading is given by the projection $M \oplus \ZZ \onto \ZZ$.
In symplectic geometry $Q$ is the moment polytope of the polarisation of $\PP(Q)$ induced by $D_Q$.
The divisor $D_Q$ is Cartier precisely when the vertices of $Q$ are lattice points.

\medskip

Now we present the combinatorial avatars of toric Fano varieties; more details on these notions can be found in \cite[\S8.3]{cox_toric_varieties} and \cite{kasprzyk_nill_fano_polytopes}. 

If $\Sigma$ is a fan in $N$, then we denote by $\Sigma(1)$ the set of the primitive generators of the rays (i.e.\ $1$-dimensional cones) of $\Sigma$.
If $\Sigma$ is a fan in $N$, then $X_\Sigma$ is Fano if and only if $\Sigma(1)$ is the set of vertices of a polytope in $N$.
If this is the case and this polytope is denoted by $P$, then we say that $\Sigma$ is the \emph{face fan} (or \emph{spanning fan}) of $P$. Such a polytope is a \emph{Fano polytope}, i.e.\ $P \subset N_\RR$ is a full dimensional lattice polytope, the origin lies in the interior of $P$ and every vertex of $P$ is a primitive lattice vector. In fact, there is a $1$-to-$1$ correspondence between Fano polytopes and toric Fano varieties.

If $P\subset N_\RR$ is a Fano polytope in the lattice $N$, its \emph{polar} is the rational polytope
\[
P^\circ := \{m \in M_\RR \mid \forall p \in P, \langle m, p \rangle \geq -1 \}\subset M_\RR. 
\]
The normal fan of $P^\circ$ is the face fan of $P$ and the divisor $D_{P^\circ}$ associated to $P^\circ$ is the toric boundary, i.e.\ the sum of the torus invariant prime divisors, and is an anticanonical divisor.

The polytope $P$ is called \emph{reflexive} if the vertices of $P^\circ$ are lattice points. We see that the toric Fano variety associated to the face fan of the Fano polytope $P$ is Gorenstein if and only if $P$ is reflexive.

\medskip

The K-polystability of Gorenstein toric Fano varieties is easy to check combinatorially. 

\begin{theorem}[{\cite[Corollary~1.2]{berman_polystability}}] \label{toricKps}
Let $P$ be a reflexive Fano polytope and let $X$ be the toric Fano variety associated to the face fan of $P$.
Then $X$ is K-polystable if and only if $0$ is the barycentre of the polar polytope $P^\circ$.
\end{theorem}
\begin{remark} Note that the same criterion holds without the reflexive assumption. This can be shown by taking the polarisation given by the multiple $-rK$, for $r$ the Gorenstein index, instead of $-K$.  
\end{remark}
Lastly, as we are interested in K-moduli and hence in automorphism groups, we mention a result determining the automorphism group of a toric Fano variety:

\begin{proposition} \label{prop:automorphism_toric_fano}
    Let $P\subset N$ be a Fano polytope and $\Aut(P)$ the finite subgroup of $\GL(N)$ consisting of the automorphisms of $P$. Let $X$ be the toric Fano variety associated to the face fan of $P$.
    
    If no facet of the polar of $P$ has interior lattice points, then $\Aut(X)$ is the semidirect product of the torus $T_N = N \otimes_\ZZ \CC^*$ with $\Aut(P)$.
\end{proposition}

The semidirect product structure is given by the embedding $\Aut(P) \into \GL(N)$.

\begin{proof}
    Let $\Sigma$ be the face fan of $P$.
    An element $m \in M$ is said to be
    a \emph{Demazure root} of $\Sigma$
     if there exists $v \in \Sigma(1)$ such that
    \begin{enumerate}[label=$\bullet$]
        \item $\langle m, v \rangle = -1$,
        \item $\forall v' \in \Sigma(1) \setminus \{v\}, \ \langle m, v' \rangle \geq 0$.
    \end{enumerate}
    Let $\cR$ denote the set of Demazure roots of $\Sigma$; then $\cR$ controls the difference between the torus $T_N$ and the connected component of the identity in the automorphism group $\Aut(X)$ of $X$ \cite{cox_jag, bruns_polytopal} (see also \cite[\S3]{nill_reductive}).
    
    Our assumption on $P^\circ$ forces $\cR = \emptyset$, hence $\Aut(X)$ is the semidirect product of the torus $T_N = N \otimes_\ZZ \CC^*$ with the automorphism group of the fan $\Sigma$, which in our case coincides with $\Aut(P)$.
\end{proof}

\section{An obstructed K-polystable toric Fano $3$-fold}
\label{sec:obstructed_toric_Kpolystable}

In this section, we construct a K-polystable toric Fano $3$-fold $X$ with Gorenstein canonical singularities that admits $3$ different smoothings. More precisely, we show:

\begin{theorem} \label{thm:example}
Let $X$ be the toric variety associated to the normal fan of the polytope $Q$ defined in \S\ref{sec:polytopesQPand3foldX}. Then:
    \begin{enumerate} [label=(\roman*)]
        \item $X$ is a Gorenstein canonical Fano $3$-fold of degree $({-}K_X)^3=12$;
        \item $X$ is K-polystable;
        \item the miniversal deformation base space of $X$ is \[\CC \pow{t_1,\dots, t_{24}} / (t_1 t_2, t_1 t_3, t_4 t_5, t_4 t_6),\]
    in particular, it has $4$ irreducible components;
        \item $X$ admits smoothings to $3$ distinct Fano $3$-folds of degree $12$; more precisely:
        \begin{enumerate}
            \item on the $22$-dimensional component $(t_1=t_4=0)$,  $X$ deforms to a Picard rank $2$ Fano in the family \morimukai{2}{6}, i.e.~to a $(2,2)$-divisor on $\PP^2\times \PP^2$, or to a double cover of a $(1,1)$-divisor on $\PP^2\times \PP^2$ branched along an anticanonical section. 
            \item on each of the $21$-dimensional components $(t_1=t_5=t_6=0)$ and $(t_2=t_3=t_4=0)$, $X$ deforms to a prime Fano $3$-fold of genus $7$ $V_{12}\subset \PP^8$. 
            \item on the $20$-dimensional component $(t_2=t_3=t_5=t_6=0)$ $X$ deforms to a Picard rank $3$ Fano in the family \morimukai{3}{1}, i.e.~to a double cover of $\PP^1 \times \PP^1 \times \PP^1$ branched over an anticanonical surface. 
         \end{enumerate}
    \end{enumerate}
\end{theorem}

\begin{remark}
	Using the Macaulay2 package \cite{ilten_package_article}, Christophersen and Ilten \cite{christophersen_ilten} compute the local structure of the Hilbert scheme of the toric $3$-fold $X$ embedded in $\PP^8$ and they show that the point corresponding to $X$ lies at the intersection of the loci of $3$ different families of smooth Fano $3$-folds embedded in $\PP^8$.
	
	We proceed by entirely different methods: we analyse the miniversal deformation base space via properties of toric Fano varieties and Altmann's work~\cite{altmann_versal}, and we identify the smoothings of $X$ via a topological analysis of vanishing cycles.
\end{remark}

\begin{remark}
	Since $X$ is K-polystable, Theorem~\ref{openness} implies that 
	    the general members of the $3$ deformation families of smooth Fano $3$-folds $V_{12}$, \morimukai{2}{6} and \morimukai{3}{1} are K-semistable.

	    This was already known: the general member of the deformation family \morimukai{2}{6} if K-semistable by Theorem~\ref{openness} and \cite{dervan_cover}, and every member of the family \morimukai{3}{1} is K-semistable by \cite{dervan_cover}.
	   A general member of the deformation family of $V_{12}$ is K-semistable \cite{ACC+}. Indeed, \cite{prok12} constructs a Fano $3$-fold $Y$ in the deformation family $V_{12}$ whose automorphism group contains a subgroup $G$ isomorphic to the simple group $\rm{SL}_2(\mathbb F_8)$. The Fano $3$-fold $Y$ is K-semistable because the $G$-equivariant $\alpha$-invariant is greater or equal to $1$ \cite{ACC+}, and by Theorem~\ref{openness}, the general element of the deformation family of $V_{12}$ is K-semistable.      
	   % \blu{What is the reference for the K-stability of $V_{12}$?}
\end{remark}

\begin{remark}
Theorem~\ref{thm:example} is consistent with the predictions of the Fanosearch program \cite{fanosearch, quantumFano3folds, sigma, petracci_roma,thomas_lagrangian}.
For all $p,q \in \ZZ$ define Laurent polynomials by
    \[
    f_{p,q} = \left( x + xy^{-1} + y + x^{-1} + x^{-1} y + y^{-1} \right) \left( z  + 2 + z^{-1} \right) + pz + q z^{-1}.
    \]
Then, the $4$ irreducible components of the miniversal deformation base space of $X$ are expected to correspond to the Laurent polynomials $f_{2,2}$ (mirror to \morimukai{2}{6}), $f_{2,3}$ and $f_{3,2}$ (both mirror to $V_{12}$) and $f_{3,3}$ (mirror to \morimukai{3}{1}).
\end{remark}

\subsection{Definition and first properties of $X$}
\label{sec:polytopesQPand3foldX}

Let $Q$ be the convex hull of 
\begin{align}\label{coords}
    x_1 &= (1,0,0)  &   x_4 &= (-1,0,0)  \nonumber \\
    x_2 &= (1,1,0)  &  x_5 &= (-1,-1,0) \\
    x_3 &= (0,1,0)  &  x_6 &= (0,-1,0) \nonumber \\
    y_1 &= (0,0,1)  & y_2 &= (-1,0,0) \nonumber
\end{align}
 in $M= \ZZ^3$. These $8$ points are precisely the vertices of $Q$.
 The lattice points of $Q$ are its vertices and the origin $y_0 = (0,0,0)$. The polytope $Q$ is the union of two hexagonal pyramids glued along their hexagonal facets (see Figure~\ref{fig:polytopes_P_Q}).

\begin{figure}
    \centering
    \includegraphics[width=\textwidth, scale = 0.75]{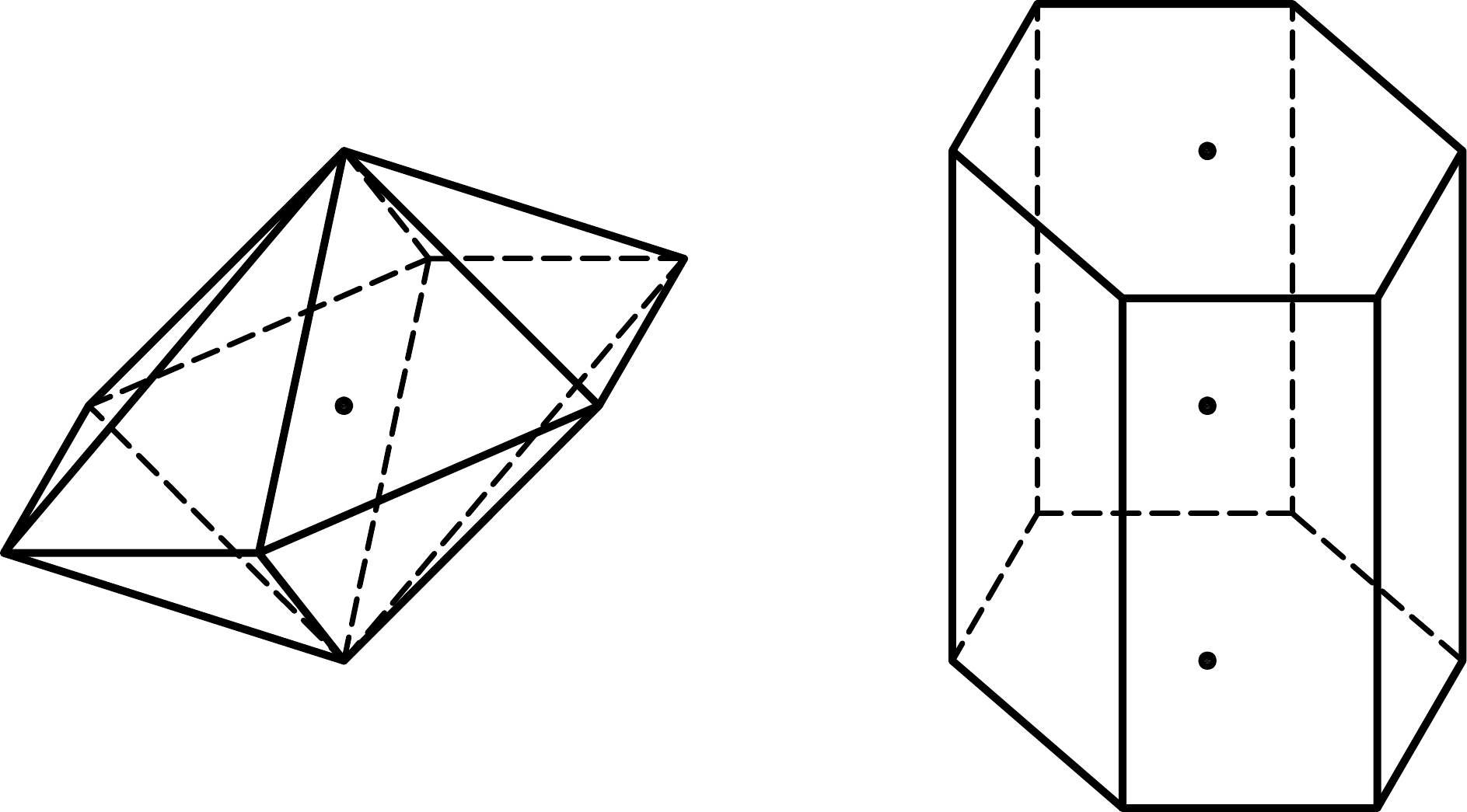}
    \caption{The polytopes $Q$ (left) and $P$ (right)} %considered in \S\ref{sec:obstructed_toric_Kpolystable}
    \label{fig:polytopes_P_Q}
\end{figure}

%\begin{remark}
%	The  automorphism group of $Q$ has order $24$. In particular, there is an involution that swaps the two $6$-valent vertices of $Q$. This involution determines an automorphism of $X$ of order $2$, which swaps the two $21$-dimensional component in the miniversal deformation of $X$.
%\end{remark}

Let $P$ be the polytope in the lattice $N= \Hom(M,\ZZ)$ with vertices: 
\begin{equation*}
\begin{pmatrix}
    1 \\ 0 \\ \pm 1
\end{pmatrix},
\begin{pmatrix}
    0 \\ 1 \\ \pm 1
\end{pmatrix},
\begin{pmatrix}
    -1 \\ 1 \\ \pm 1
\end{pmatrix},
\begin{pmatrix}
    -1 \\ 0 \\ \pm 1
\end{pmatrix},
\begin{pmatrix}
    0 \\ -1 \\ \pm 1
\end{pmatrix},
\begin{pmatrix}
    1 \\ -1 \\ \pm 1
\end{pmatrix}
\end{equation*}
The polytope $P$ is a hexagonal prism (see Figure~\ref{fig:polytopes_P_Q}).
One sees that $Q$ is the polar of $P$, so $P$ is reflexive (its reflexive ID in \cite{grdb} is 3875). 

The toric variety $X$ associated to the normal fan of $Q$ (face fan of $P$) is an anticanonical degree $12$ Fano $3$-fold with Gorenstein canonical singularities \cite{Kas08}.
Since the barycentre of $Q$ is the origin, $X$ is K-polystable by Theorem~\ref{toricKps}.
Assertions (i) and (ii) in Theorem~\ref{thm:example} are proved.

\subsection*{Explicit description of the embedding $X\subset \PP^8$}

By the construction recalled in \S\ref{sec:tv}, as $Q$ is the moment polytope of ${-}K_X$, the anticanonical ring satisfies 
\[ R(X,{-}K_X)\simeq \CC[ \cone{Q \times \{1\}} \cap (M \oplus \ZZ)   ],
\]
where the semigroup algebra is $\NN$-graded by the projection $M \oplus \ZZ \to \ZZ$, and we check that $R(X,{-}K_X)$ is generated in degree $1$. The anticanonical map $\Phi_{|{-}K_X|}$ is a closed embedding into $\PP^8$, whose image is defined by the quadratic binomial equations:
\begin{gather*}
  \mathrm{rank} \begin{pmatrix}
      y_0 & x_1 & x_2 \\
      x_4 & y_0 & x_3 \\
      x_5 & x_6 & y_0
  \end{pmatrix}  \leq 1 \qquad \text{and} \qquad 
  y_0^2 - y_1 y_2 = 0.
    \end{gather*}

\subsection*{Singular locus of $X$}\label{sec:sing}
Consider the open subschemes of $X$ defined by:
\begin{align*}
	Y_j &= X \cap \{y_j \neq 0\} &  &\text{for } j =1,2, \\
    U_i &= X \cap \{ x_i \neq 0 \} & &\text{for } i \in \ZZ / 6 \ZZ,  \\
    U_{i,i+1} &= U_i \cap U_{i+1} & & \text{for } i \in \ZZ / 6 \ZZ.
    \end{align*}
In other words, $Y_j$ for $j=1,2$ are the affine charts associated to the cones over the two hexagonal facets of $P$, $U_i$ for $i=1,\dots, 6$ are the charts associated to the cones over the rectangular facets of $P$, and $U_{i,i+1}$ are the charts associated to the cones over the lattice length $2$ edges of $P$. These charts cover $X$:\[
X = Y_1 \cup Y_2 \cup U_1 \cup U_2 \cup U_3 \cup U_4 \cup U_5 \cup U_6.
\]
Let $p_i\in U_i$ for $i= 1, \cdots, 6$, and $q_j\in Y_j$  for $j =1,2$ be the torus fixed points, and denote by $U$ the open non-affine subscheme of $X$  defined by \[U = X \setminus \{q_1, q_2 \} = U_1 \cup \dots \cup U_6. \]

The points $q_1$ and $q_2$ are isolated singularities; they are both isomorphic to the vertex of the cone over the anticanonical embedding of the smooth degree $6$ del Pezzo surface $S_6\subset \PP^6$.
This singularity is Gorenstein and strictly canonical and was studied by  \cite{altmann_minkowski, altmann_versal}; it admits two topologically distinct smoothings (see \S\ref{sec:local_smooth}).

The singular locus of $X$ consists of $q_1$ and $q_2$ and of a $1$-dimensional connected component $\Gamma$:
\[
\Sing X= \{q_1, q_2\} \sqcup \Gamma.
\]
The curve $\Gamma$ is a cycle of $6$ smooth rational curves and is the underlying topological space of the closed subscheme of $U$ defined by the $3$rd Fitting ideal of $\Omega^1_U$. 

We now describe the structure of $X$ in a neighbourhood of $p_i \in \Gamma$. The situation is the same up to reordering of the coordinates for all $i\in \ZZ/6\ZZ$, so we fix $i=1$. 
As $U_1= \{x_1\neq 0\}\cap X$, setting $x_1=1$ in the equations of $X$ gives:
\begin{equation}\label{eqU1}
U_1= \begin{cases}		y_0 = x_2x_6 \\
			x_3 =y_0x_2 \\
	x_4 = y_0^2\\
	x_5 = y_0x_6.
	\end{cases}\subset \AA^8\end{equation}
	so that denoting by $x=y_1$, $y= y_2$, $z= x_2$, $t=x_6$, we find that $U_1$ is isomorphic to
\begin{equation*}
V= \Spec \CC \! \left[ x,y,z,t\right] \left/ \left( xy -z^2t^2\right) \right. \! .
\end{equation*}
The hypersurface $V\subset \AA^4$ is an affine toric $3$-fold with canonical Gorenstein singularities and its singular locus $\Sing V$ has two irreducible components $(x = y = z = 0)$ and $(x = y= t = 0)$.
Generically on these two components $V$ is locally isomorphic to a product $\GG_\rmm \times (\mbox{$2$-dimensional ordinary double point})$.
Explicitly the open subscheme $V \cap \{ t \neq 0 \}$ is isomorphic to $(\GG_\rmm)_t \times A_1\subset \AA^4_{x,y,z,t}$, where $A_1 = \Spec \CC[x,y,z] / (xy-z^2)$.

Since $U_{1,2}= U_1\cap \{x_2\neq 0\}$ (resp.~$U_{1,6}= U_1\cap \{ x_6\neq 0\}$), setting $x_2=1$ (resp.~$x_6=1$) in \eqref{eqU1}, shows that $U_{1,2}$  and $U_{1,6}$ are isomorphic to
\[ \{ xy-t^2=0\} \subset (\GG_\rmm)_z \times  \AA^3_{x,y,t} \quad \mbox{and} \quad \{ xy-z^2=0\} \subset (\GG_\rmm)_t \times \AA^3_{x,y,z}.
\]

%
%\begin{comment}We now denote by $\Gamma$ the closed subscheme of $X$ defined by the $3$rd Fitting ideal of  subscheme of $\PP^8$ defined by the homogeneous ideal generated by $y_0^2$, $y_1$, $y_2$, $y_0 x_i x_{i+1}$ for all $i \in \ZZ / 6 \ZZ$, and by the $2 \times 2$-minors of the matrix
%\begin{equation*}
%\begin{pmatrix}
%        y_0 & x_1 & x_2 \\
%        x_4 & y_0 & x_3 \\
%        x_5 & x_6 & y_0
%    \end{pmatrix}.
%\end{equation*}
%It is clear that $\Gamma$ is a non-reduced closed subscheme of $X$. Let $\Gamma_\red$ be the reduced structure of $\Gamma$. The ideal of $\Gamma_\red$ in $\PP^8$ is generated by $y_0$, $y_1$, $y_2$, and by the $2 \times 2$-minors of the matrix above.
%\end{comment}

\subsection{Local smoothings of singularities}\label{sec:local_smooth}

\subsection*{Smoothing of the isolated singular points.}
\label{sec:cone_dP6}
Let $q\in Y$ be the vertex of the affine cone over the anticanonical embedding of the smooth del Pezzo surface of degree $6$ in $\PP^6$. 
Altmann \cite{altmann_minkowski, altmann_versal} shows that the base of the miniversal deformation of $Y$ is:
\begin{equation*}
   \Def Y \simeq \mathrm{Spf} \ \CC \pow{s_1, s_2, s_3} / (s_1 s_2, s_1 s_3).
\end{equation*}
In particular, $\Def Y$ has two irreducible components: one has dimension $2$, the other has dimension $1$ and both yield smoothings \cite{petracci_roma}.
The Betti numbers of the Milnor fibres of these smoothings are computed in the next Proposition.

\begin{proposition} \label{prop:milnor_fibres_smoothings_hexagon}
    Let $Y$ be the affine cone over the anticanonical embedding of the smooth del Pezzo surface of degree $6$ in $\PP^6$ and let $\Def Y$ be the base of the miniversal deformation of $Y$. Let $M_j$ be the Milnor fibre of the smoothing given by a general arc in the $j$-dimensional component of $\Def Y$ for $j=1,2$. The Betti numbers of $M_j$ are:
    \begin{eqnarray*}
    b_2(M_2) = 1, \qquad b_3(M_2) = 2  \qquad \text{and} \qquad b_i(M_2) = 0 \text{ if } i = 1 \text{ or } i \geq 4;
   \\
     b_2(M_1) = 2, \qquad b_3(M_1) = 1 \qquad \text{and} \qquad b_i(M_1) = 0 \text{ if } i = 1 \text{ or } i \geq 4.
\end{eqnarray*}

\end{proposition}

\begin{proof}
The Milnor fibre of an arbitrary holomorphic map germ from a $4$-dimensional complex analytic space to $(\CC, 0)$ has the homotopy type of a finite CW complex of dimension $\leq 3$ \cite{Milnor}, so $b_i(M_1)= b_i(M_2)=0$ for all $i\geq 4$. Further, $b_1(M_1)= b_1(M_2)=0$ by \cite{greuel_steenbrink}. 
       
    Let $Z$ be the projective cone over $S_6 \subset \PP^6$. Then, $Z$ is a toric variety by \cite{petracci_roma}, and $Z$ contains $Y$ as an open subscheme. The singular locus of $Z$ consists of an isolated singularity lying on $Y$ (the vertex of the cone). By Proposition~\ref{prop:jordan}, the Betti numbers of $Z$ are $b_2(Z)  = 1$, $b_3(Z) = 0$ and $b_4(Z) = 4$.
    
  By \cite[Proposition~2.2]{petracci_roma} the restriction morphism $\Def Z \to \Def Y$ is smooth and induces an isomorphism on tangent spaces, hence the base of the miniversal deformation of $Z$ is the germ at the origin of $\left( s_1 s_2 = s_1 s_3 = 0 \right)$ in $\CC^3$.
    
    Let $\cZ \to \Delta$ be a smoothing of $Z$ given by a general arc in the component $(s_1 = 0)$ and let $M_2$ be the corresponding Milnor fibre.
 \cite[Proposition~2.2]{petracci_roma} shows that $\cZ_t$ is a divisor in $\PP^2 \times \PP^2$ of type $(1,1)$, hence $b_2(\cZ_t) = b_4(\cZ_t)= 2$ and $b_3(\cZ_t) = 0$. By the long exact sequence ~\eqref{les}, we have: 
    \begin{gather*}
        0 \to \coh{2}{Z} = \QQ \to \coh{2}{\cZ_t} = \QQ^2 \to \coh{2}{M_2} \to \coh{3}{Z} = 0 \\
        \coh{3}{\cZ_t} = 0  \to \coh{3}{M_2} \to \coh{4}{Z} = \QQ^4 \to \coh{4}{\cZ_t} = \QQ^2 \to \coh{4}{M} =0
    \end{gather*}
 so that $b_2(M_2) = 1$ and $b_3(M_2) = 2$.

Now consider $M_1$, the Milnor fibre associated to the smoothing $\cZ \to \Delta$ of $Z$ given by the component $\left(s_2 = s_3 =0\right)$. \cite[Proposition~2.2]{petracci_roma} shows that $\cZ_t$ is isomorphic to $\PP^1 \times \PP^1 \times \PP^1$, so that $b_2(\cZ_t) = b_4(\cZ_t) = 3$ and $b_3(\cZ_t) = 0$. The same argument as above yields $b_2(M_1)= 2$ and $b_3(M_1)=2$.\end{proof}

\subsection*{Local smoothings of the $1$-dimensional component of $\Sing V$} \label{sec:hypersurface_V}
Recall that each $U_i$ is isomorphic to the hypersurface $V= \{ xy-z^2t^2=0\}\subset \AA^4$.
We obtain a free resolution of $\Omega ^1_V$ 
\begin{equation}\label{conormal}
0 \longrightarrow \cO_V \xrightarrow{\begin{pmatrix}
    -2 z t^2 \\
    -2 z^2 t \\
    y \\
    x
    \end{pmatrix}} \cO_V^{\oplus 4} \longrightarrow \Omega^1_V \longrightarrow 0.
\end{equation}
by considering the conormal sequence of $V \into \AA^4 \to \Spec \CC$ . 
Applying the functor $\Hom_{\cO_V}(\cdot, \cO_V)$ to \eqref{conormal}, we get: 
\begin{equation*}
\TT^1_V = \frac{\cO_V}{(zt^2, z^2 t, x, y)} = \frac{\CC[x, y,z,t]}{(zt^2, z^2 t, x, y)},
\end{equation*}
so that $\TT^1_V$ is isomorphic to $\cO_{\Gamma}$, where $\Gamma$ is the scheme defined by the $3$rd Fitting ideal of $\Omega^1_V$, which has an embedded point at the origin. The underlying topological space of $\Gamma$ is the singular locus of $V$, with ideal $(x,y,zt)$. 
The $\CC$-vector space $\TT^1_V$ has a homogeneous basis
$
\{1, z,t \} \cup \{z^n \mid n \geq 1 \} \cup \{t^n \mid n \geq 1 \}.
$

Since $V$ is a hypersurface, $\TT^2_V = 0$ and its deformations are unobstructed. It is easy to see that the $1$-parameter deformation $(xy-z^2t^2 + \lambda = 0)$, associated to the section $1$ of $\TT^1_V$, gives a smoothing of $V$.

\begin{remark}
We denote by the same name the schemes defined by the third Fitting ideal of $\Omega^1_V$ and by the third Fitting ideal of $\Omega^1_U$. Strictly speaking, one is a local version of the other; we trust this will not lead to any confusion. 
\end{remark}

We now see how these local smoothings combine to define a smoothing of $U$. 
\subsection*{Deformations of $U_i$}
As above, fix $i=1$.
From the above, $\cT^1_{U_1}= \cT^1_{V}$ is isomorphic to the structure sheaf of $\Gamma_{\vert U_1}$ and $\TT^2_{U_1} = 0$.

Since the torus $T_N = \Spec \CC[M]$ acts on $U_1$, $\TT^1_{U_1}$ is $M$-graded. For every $m \in M$ let $\TT^1_{U_1}(m)$ denote the eigenspace with respect to the character $m$.
We have
\begin{equation*}
\dim \TT^1_{U_1}(m) =
\begin{cases}
	1 &\qquad \text{if } m=(1,0,0) \text{ or } m = (2,0,0), \\
	1 &\qquad \text{if } m=(2,k,0) \text{ with } k \in \ZZ, k \geq 1, \\
	1 &\qquad \text{if } m=(2-k,-k,0) \text{ with } k \in \ZZ, k \geq 1, \\
	0 &\qquad \text{otherwise.}
	\end{cases}
\end{equation*}
Some of these degrees are depicted on the left in Figure~\ref{fig:tre_esagoni}. We have shown:

\begin{figure}
	\centering
	\includegraphics[width=\textwidth, scale = 0.75]{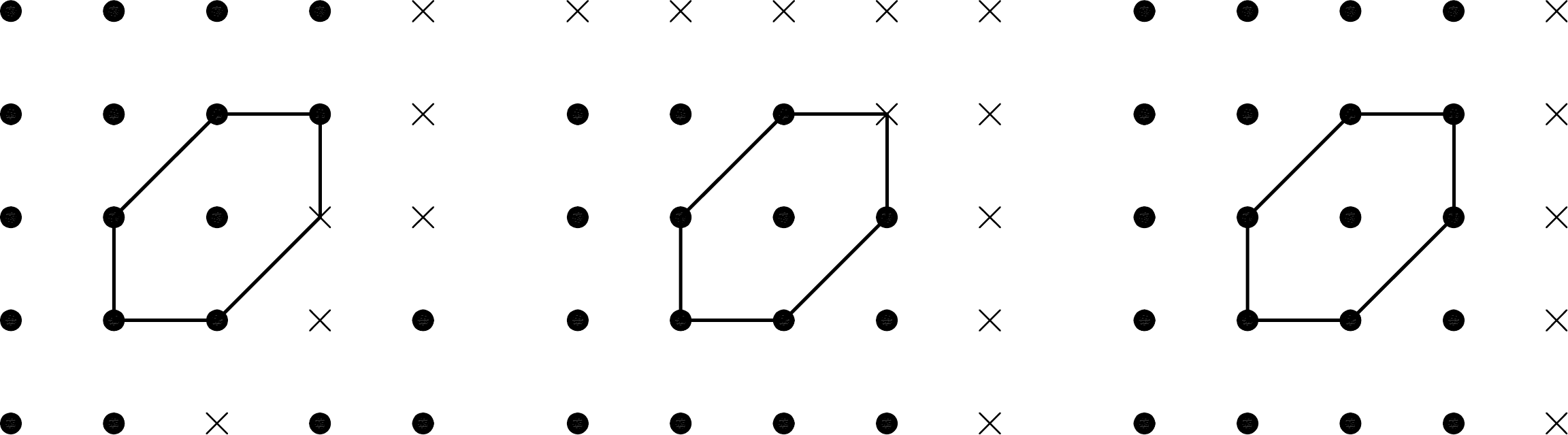}
	\caption{Some of the degrees where $\TT^1_{U_1}$ (left), $\TT^1_{U_2}$ (centre) and $\TT^1_{U_{1,2}}$ (right) are non-zero}
	\label{fig:tre_esagoni}
\end{figure}

\begin{lemma}
	For every $i \in \ZZ / 6 \ZZ$, $U_i\simeq V$, where $V$ is defined in \S\ref{sec:hypersurface_V}, $\TT^2_{U_i} = 0$, and
	\begin{equation*}
		\dim \TT^1_{U_i}(m) =
		\begin{cases}
			1 &\qquad \text{if } m=x_i \text{ or } m = 2 x_i, \\ 
			1 &\qquad \text{if } m=2x_i + k(x_{i+1} - x_i) \text{ with } k \in \ZZ, k \geq 1, \\
			1 &\qquad \text{if } m=2x_i + k(x_{i-1} - x_i) \text{ with } k \in \ZZ, k \geq 1, \\
			0 &\qquad \text{otherwise.}
		\end{cases}
	\end{equation*}
Here $x_i$ are the lattice points of $M$ defined in \eqref{coords}.%\S\ref{sec:polytopesQPand3foldX}.
\end{lemma}

Some of the degrees of $\TT^1_{U_2}$ are depicted in the middle of Figure~\ref{fig:tre_esagoni}.

\subsection*{Deformations of $U_{i,i+1}$}
We have seen that $U_{1,2}\simeq\GG_\rmm \times A_1$, where $A_1 = \Spec \CC[x,y,t] / (xy-t^2)$. The homogeneous summands of $\TT^1_{U_{1,2}}$ are as follows:
\begin{equation*}
	\dim \TT^1_{U_{1,2}}(m) =
	\begin{cases}
		1 &\qquad \text{if } m=(2,k,0) \text{ with } k \in \ZZ, \\
		0 &\qquad \text{otherwise.}
	\end{cases}
\end{equation*}
Some of these degrees are depicted on the right in Figure~\ref{fig:tre_esagoni}.
As above, the situation is the same for all $i \in \ZZ/6\ZZ$, so that:

\begin{lemma}
	For every $i \in \ZZ / 6 \ZZ$, $\TT^2_{U_{i,i+1}} = 0$ and
	\begin{equation*}
		\dim \TT^1_{U_{i,i+1}}(m) =
		\begin{cases}
			1 &\qquad \text{if } m=2x_i + k(x_{i+1} - x_i) \text{ with } k \in \ZZ, \\
			0 &\qquad \text{otherwise.}
		\end{cases}
	\end{equation*}
	Here $x_i$ are the lattice points of $M$ defined in \eqref{coords}.
\end{lemma}

\subsection*{Deformations of $U$}

\begin{lemma}\label{lem:H0_H1_U}
    $\rH^0(U, \cT^1_U) = \CC^{18}$ and $\rH^1(U, \cT^1_U) = 0$.
    \end{lemma}

\begin{proof}
	If $j \notin \{ i, i\pm 1 \}$,  then $U_i \cap U_j$ is smooth, so the restriction of $\cT^1_U$ to $U_i \cap U_j$ is zero. It follows that the \v{C}ech complex of the coherent sheaf $\cT^1_U$ with respect to the affine cover $\{U_i \mid i \in \ZZ / 6 \ZZ \}$ is concentrated in degrees $0$ and $1$:
	\begin{equation*}
	\bigoplus_{i \in \ZZ / 6  \ZZ} \TT^1_{U_i} \overset{d}\longrightarrow 
	\bigoplus_{i \in \ZZ / 6  \ZZ} \TT^1_{U_{i,i+1}}.
	\end{equation*}
As this is a homomorphism of $M$-graded vector spaces, we analyse the homogeneous components of this complex for every degree $m \in M$.
	\begin{enumerate}[label=$\bullet$]
		\item If $m=x_i$ is a vertex of the hexagon, the complex is $\CC \to 0$.
		\item If $m = 2x_i$, the complex is
		\begin{equation*}
		\CC^3 \xrightarrow{\begin{pmatrix}
				-1 & 1 & 0 \\ 0 & -1 & 1
		\end{pmatrix}} \CC^2.
		\end{equation*}
	\item If $m=x_i + x_{i+1}$, the complex is
	\begin{equation*}
		\CC^2 \xrightarrow{\begin{pmatrix}
				-1 & 1 
		\end{pmatrix}} \CC^1.
	\end{equation*}
\item If $m = (2-k) x_i + k x_{i+1}$ with $k \in \ZZ$ and $|k| \geq 3$, the complex is
\begin{equation*}
\CC \overset{\pm 1}\longrightarrow \CC.
\end{equation*}
\item In all the other cases, the complex is $0 \to 0$.
	\end{enumerate}

From $\rH^0(U,\cT^1_U) = \ker d$ we deduce
\begin{equation*}
\dim \rH^0(U, \cT^1_U)(m) = \begin{cases}
1 &\qquad \text{if } m=x_i \text{ for some } i \in \ZZ / 6 \ZZ, \\
1 &\qquad \text{if } m=2x_i \text{ for some } i \in \ZZ / 6 \ZZ, \\
1 &\qquad \text{if } m=x_i+x_{i+1} \text{ for some } i \in \ZZ / 6 \ZZ, \\
0 &\qquad \text{otherwise.}
\end{cases}
\end{equation*}
The $18$ degrees of $M$ in which $\rH^0(U,\cT^1_U)\neq 0$ are depicted in
Figure~\ref{fig:esagono_solo}.
Since the differential $d$ is surjective, we have $\rH^1(U, \cT^1_U) = \coker d = 0$.
\end{proof}

\begin{figure}
	\centering
	\includegraphics[width=5cm, scale = 0.75]{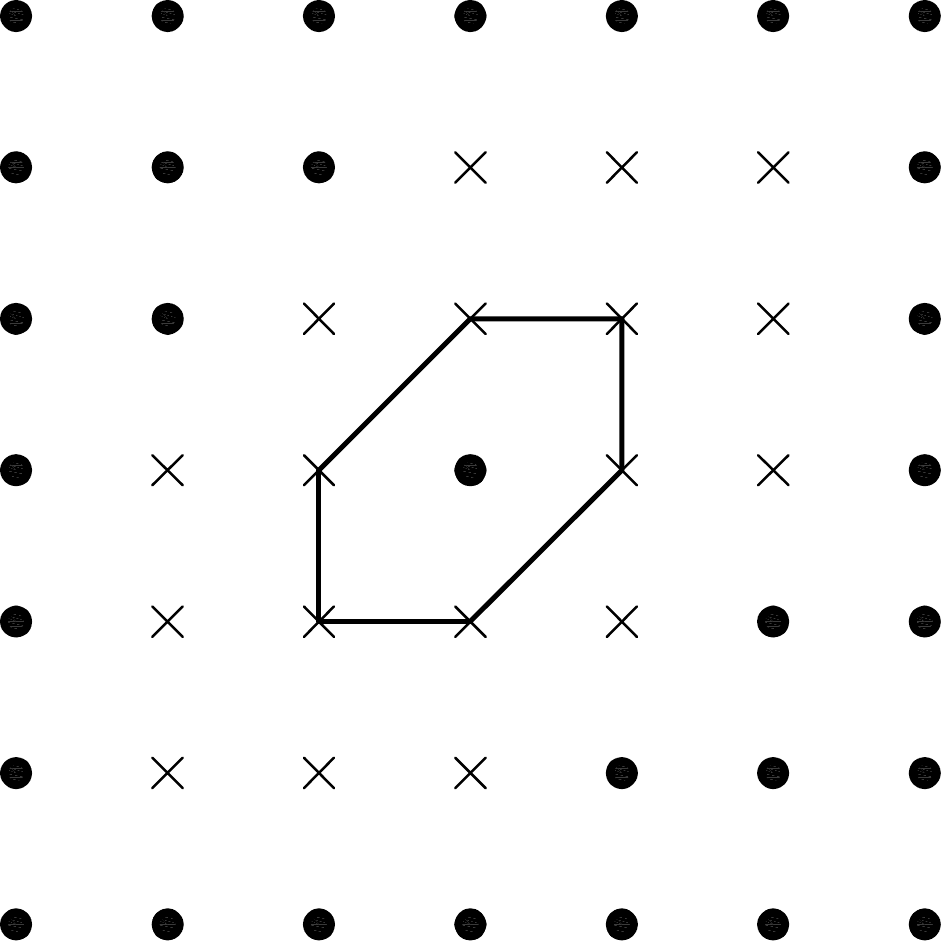}
	\caption{The $18$ degrees in which $\rH^0(U, \cT^1_U)\neq 0$ }
	\label{fig:esagono_solo}
\end{figure}

\begin{remark} 
We have seen that $\cT^1_U$ is a line bundle on the non-reduced scheme $\Gamma$. One can show that
	\[\cT^1_U = \cO_\Gamma(2)\]
		where 
	$\cO_X(1) = \omega_X^\vee$ defines the embedding in $\PP^8$.
    For $i \in \ZZ / 6 \ZZ$ let $\Gamma_{i,i+1}$ be the irreducible component of $\Gamma$ given by the closure of $\Gamma \cap U_{i,i+1}$.
    Let $\Gamma_\mathrm{red}$ denote the reduced structure of $\Gamma$.
    Tensoring the two short exact sequences
	\begin{equation*}
		0 \longrightarrow \bigoplus_{i \in \ZZ / 6 \ZZ} \kappa(p_i) \longrightarrow \cO_\Gamma \longrightarrow \cO_{\Gamma_\mathrm{red}} \longrightarrow 0
	\end{equation*}
and
	\begin{equation*}
	0 \longrightarrow   \cO_{\Gamma_\mathrm{red}} \longrightarrow \bigoplus_{i \in \ZZ/6\ZZ} \cO_{\Gamma_{i,i+1}} \longrightarrow \bigoplus_{i \in \ZZ / 6 \ZZ} \kappa(p_i) \longrightarrow 0
\end{equation*}
with $\cO_X(2)$, one computes the dimensions of $\rH^0(U, \cT^1_U)$ and of $\rH^1(U, \cT^1_U)$. This gives another proof of Lemma~\ref{lem:H0_H1_U}.
\end{remark}

\subsection{Smoothings of $X$ --- proof of Theorem~\ref{thm:example}(iii, iv)}
Recall that $\Sing X$ has $3$ connected components: the points $(q_1 \in Y_1)$ and $(q_2 \in Y_2)$ and the reducible curve $\Gamma \subset U$.
Since $U$ is lci, $\cT^2_U = 0$. We have  
\begin{align*}
\cT^1_X &= \cT^1_U \oplus \TT^1_{Y_1} \oplus \TT^1_{Y_2}, \\
\cT^2_X &= \TT^2_{Y_1} \oplus \TT^2_{Y_2}.
\end{align*}
Lemma~\ref{lem:H0_H1_U} shows that
\begin{align*}
	\rH^0(X, \cT^1_X) &= \CC^{18} \oplus \TT^1_{Y_1} \oplus \TT^1_{Y_2}, \\
	\rH^1(X, \cT^1_X) &= 0.
    \end{align*}
Moreover
\[	\rH^0(X, \cT^2_X) = \TT^2_{Y_1} \oplus \TT^2_{Y_2}.\]
As $X$ is a toric Fano variety, $\rH^1(X, \cT^0_X) = 0$ and $\rH^2(X, \cT^0_X) = 0$ (see \cite[Proof of Theorem~5.1]{totaro} and \cite[Lemma~4.4]{petracci_survey}).
From the spectral sequence 
\begin{equation*}
E_2^{p,q} = \rH^p(X, \cT^q_X) \Longrightarrow \TT^{p+q}_X
\end{equation*}
we deduce
\begin{align*}
	\TT^1_X &= \rH^0(X, \cT^1_X) = \CC^{18} \oplus \TT^1_{Y_1} \oplus \TT^1_{Y_2}, \\
	\TT^2_X &= \rH^0(X, \cT^2_X) = \TT^2_{Y_1} \oplus \TT^2_{Y_2},
	\end{align*}
and hence the product of the restriction morphisms
\begin{equation*}
\Def X \longrightarrow \Def Y_1 \times \Def Y_2
\end{equation*}
is smooth of relative dimension $18$.
The discussion of \S\ref{sec:cone_dP6} concludes the proof of Theorem~\ref{thm:example}(iii).

Now we  determine the general fibres of the smoothings of $X$. 

\begin{lemma} \label{lem:smoothing_of_X}
Each irreducible component of the base space of miniversal deformations of $X$ yields a smoothing of $X$.
\end{lemma}

\begin{proof}
	The six degrees $2x_i \in M$, for $i \in \ZZ/6\ZZ$, of $\rH^0(U, \cT^1_U)$ can be combined to give a deformation of $X$ over a smooth $6$-dimensional germ because they are unobstructed. As noted in \S\ref{sec:hypersurface_V} this defines a smoothing of $U$.
 Combining this deformation of $X$, which smoothes $U$, with the smoothings of the two isolated singularities $(q_j \in Y_j)$, $j=1,2$, finishes the proof.
\end{proof}

In this section, we use topological properties of deformations to determine which Fano $3$-folds appear as general fibres in the three distinct smoothings above. 
Let $f\colon \cX \to \Delta$ be a smoothing of $X$; then for any $t\in \Delta$, $\cX_t$ is a smooth Fano $3$-fold of anticanonical degree $({-}K_X)^3= 12$. By the classification of smooth Fano $3$-folds \cite{Isk1, Isk2,MM82, MM03}, there are precisely $5$ deformation families of smooth Fano $3$-folds with anticanonical degree $12$:
\begin{enumerate}[label=$\bullet$]
\item $V_{12}\subset \PP^8$, the prime Fano $3$-fold of genus $7$, with $\chi=-10$,
\item \morimukai{2}{5}, the blowup of $V_3\subset \PP^4$ along a plane cubic lying on it, with $\chi= -6$,
\item \morimukai{2}{6}, either a $(2,2)$-divisor on $\PP^2\times \PP^2$, or a double cover of a $(1,1)$-divisor on $\PP^2\times \PP^2$ branched along an anticanonical section, with $\chi=-12$,
\item \morimukai{3}{1}, a double cover of $\PP^1\times \PP^1\times \PP^1$ branched along an anticanonical section, $\chi= -8$. 
\item $\PP^1\times S_2$, where $S_2$ is a smooth del Pezzo surface of degree $2$. 
\end{enumerate}

\begin{remark}
By Kawamata--Viehweg vanishing, the anti-plurigenera are deformation invariant, and very ampleness of the anticanonical divisor is an open condition. In particular, $\PP^1 \times S_2$ cannot arise as a general fibre of a smoothing of $X$ because its anticanonical line bundle is not very ample.
\end{remark}

By \eqref{topology_smoothing}, the Euler characteristics of $X$ and $\cX_t$ are related by:
\[ \chi(\cX_t)= \chi(X)+  \chi(\mathbb H^\bullet(X, \varphi_f \underline \QQ_{\cX})),\]
where $\varphi_f \underline \QQ_{\cX}$ is the sheaf of vanishing cycles of $f$ introduced in \S\ref{sec:topo_smooth}. 
Using Proposition~\ref{prop:jordan}, we compute the Betti numbers of $X$; these are $b_1(X) = b_5(X) = 0$, $b_2(X) = 1$, $b_3(X) = 4$, $b_4(X) = 9$, so that $\chi(X)= 8$. 

We now gather some information on $\chi(\mathbb H^\bullet(X, \varphi_f \underline \QQ_{\cX}))$. 
The complex $\varphi_f \underline \QQ_\cX^\bullet$ is constructible, by \cite[Appendix B.4]{Massey}. 
The sheaf of vanishing cycles is supported on the singular locus of $X$, so that by Mayer--Vietoris: 
\[ \chi (\mathbb H^\bullet (X, (\varphi_f \underline \QQ)^\bullet)) = \chi (\mathbb H^\bullet(X, (\varphi_f \underline \QQ)^\bullet_{\vert Y_1}))+ \chi (\mathbb H^\bullet(X, (\varphi_f \underline \QQ)^\bullet_{\vert Y_2}))+ \chi (\mathbb H^\bullet(X, (\varphi_f \underline \QQ)^\bullet_{\vert U})).\]
By the results of \S\ref{sec:topo_smooth} we have the following equality of Euler characteristics 
\[
\chi (\mathbb H^\bullet(X, (\varphi_f \underline \QQ)^\bullet_{\vert Y_i}))= \widetilde{\chi}(F_{Y_i}),
\]
where $\widetilde{\chi}$ denotes the reduced Euler characteristic and  $F_{Y_i}$ denotes the Milnor fibre of the induced smoothing of the singular point $Y_i$. Recall that each $Y_i$ has two possible smoothings, the Milnor fibres of which are denoted $M_1$ and $M_2$, and by Proposition~\ref{prop:milnor_fibres_smoothings_hexagon}, 
\[ \widetilde{\chi}(F_{Y_i})= \pm1.\]
It follows that for a smoothing $f\colon \cX\to \Delta$, we have 
\[ \chi(\cX_t) = \chi(X)  +\chi (\mathbb H^\bullet(X, (\varphi_f \underline \QQ)^\bullet_{\vert U})) +\begin{cases} 2\\0\\-2\end{cases}.\]

There are $3$ distinct smoothings. Indeed, $\Def X$ has $4$ components: one $22$-dimensional component, two $21$-dimensional components that are exchanged by the involution of $X$ swapping the isolated singularities $(q_1 \in Y_1)$ and $(q_2 \in Y_2)$, and one $20$-dimensional component. Let $A$ (resp.~ $B$, resp.~$C$) be a general fibre of the smoothing over the $22$ (resp.~$21$, resp.~$20$)-dimensional component of $\Def X$. 
Then $A,B,C$ are smooth Fano $3$-folds of degree $12$ with very ample anticanonical line bundle, so their Euler characteristics belong to $\{ -6,-8,-10,-12\}$.
Consider a general arc in the $22$-dimensional component of $\Def X$, i.e.\ the smoothing from $X$ to $A$. In this smoothing the Milnor fibre of the singularity $(q_j \in Y_j)$, for each $j=1,2$, is $M_2$. We thus have that: 
\[
\chi(A)= 8+ \chi (\mathbb H^\bullet(X, (\varphi_f \underline \QQ)^\bullet_{\vert U})) -2 = 6+\chi (\mathbb H^\bullet(X, (\varphi_f \underline \QQ)^\bullet_{\vert U}))
\]
and similarly, 
 on each of the $21$-dimensional components, the Milnor fibre of one of $(q_j \in Y_j)$ is $M_2$ and that of the other is $M_1$. We thus get that 
 \[
 \chi(B)= 8+ \chi (\mathbb H^\bullet(X, (\varphi_f \underline \QQ)^\bullet_{\vert U})) +0= \chi(A)+2.
 \]
 Similarly, on the $20$-dimensional components, the Milnor fibre of both $(q_j \in Y_j)$ is $M_1$ and 
 \[
 \chi (C)= 8+ \chi (\mathbb H^\bullet(X, (\varphi_f \underline \QQ)^\bullet_{\vert U})) + 2 = \chi(A)+4.
 \]
 The Euler characteristics of $A$, $B$ and $C$ are thus either $-6$, $-8$ and $-10$ or $-8$, $-10$, $-12$, respectively. In the first case, $\chi (\mathbb H^\bullet(X, (\varphi_f \underline \QQ)^\bullet_{\vert U}))=-16$, while in the second $\chi (\mathbb H^\bullet(X, (\varphi_f \underline \QQ)^\bullet_{\vert U}))=-18$. 
 
By \cite[Theorem 4.1.22]{Dimca}, if $\mathcal S$ is a Whitney regular stratification such that $\varphi_f\underline \QQ^\bullet_{|U}$ is an $\mathcal S$-constructible bounded complex on $X$, and $x_S\in S$ is an arbitrary point of the stratum $S\in \mathcal S$, then 
\[\chi (\mathbb H^\bullet(X, (\varphi_f \underline \QQ)^\bullet_{\vert U})) = \sum_{S\in \mathcal S}\chi(S)\cdot \chi(\mathcal H^\bullet( \varphi_f \underline \QQ^\bullet_{\vert U})_{x_S}).\]
The complex $\varphi_f\underline \QQ^\bullet_{|U}$ is supported on the curve $\Gamma$, and by the description in  \S\ref{sec:sing}, the local description is the same on each component of $\Gamma$. We may assume (up to taking a refinement of $\mathcal S$) that $\mathcal S$ induces a stratification of each component of $\Gamma$ and that the induced stratifications on the components of $\Gamma$ are isomorphic to each other. It follows that the number of components of $\Gamma$-- i.e.~$6$-- necessarily divides the Euler characteristic $\chi (\mathbb H^\bullet(X, (\varphi_f \underline \QQ)^\bullet_{\vert U}))$, which thus has to be $-18$. It follows that $A$ is in the family \morimukai{2}{6}, $B$ is a prime Fano $V_{12}$ and $C$ is in the family \morimukai{3}{1}.
This concludes the proof of Theorem~\ref{thm:example}.

\subsection{Local description of the K-moduli space and stack}
\label{Kmodsp}

\subsection*{The automorphism group of $X$}
Let $\Aut(P)$ be the subgroup of $\GL(N) = \GL_3(\ZZ)$ that preserves the polytope $P$. It is clear from the symmetries of $P$, that $\Aut(P)$ is generated by the involution
\begin{equation*}
    \begin{pmatrix}
        1 & 0 & 0 \\
        0 & 1 & 0 \\
        0 & 0 & -1
    \end{pmatrix},
\end{equation*}
which swaps the top facet of $P$ and the bottom facet of $P$, and
and by the dihedral group of the hexagon, itself generated by
\begin{equation*}
    \begin{pmatrix}
        1 & 1 & 0 \\
        -1 & 0 & 0 \\
        0 & 0 & 1
    \end{pmatrix}
    \qquad \text{and} \qquad
    \begin{pmatrix}
        0 & 1 & 0 \\
        1 & 0 & 0 \\
        0 & 0 & 1
    \end{pmatrix}.
\end{equation*}
Therefore, $\Aut(P)$ is a finite group of order $24$ isomorphic to $D_6 \rtimes C_2$, where $D_6$ denotes the dihedral group of order $12$ and $C_2$ the cyclic group of order $2$. 
Since no facet of $Q$ has interior lattice points, by Proposition~\ref{prop:automorphism_toric_fano}, $\Aut(X)$ is the semidirect product of $T_N = N \otimes_\ZZ \CC^* \simeq (\CC^*)^3$ with $\Aut(P)$.

\subsection*{Local reducibility of K-moduli}

Here we examine the local structure of the algebraic stack $\stack{3}{12}$ and of the algebraic space $\modspace{3}{12}$ at the closed point corresponding to $X$. Denote by $G$ the group $\Aut(X)$ and let
\[
A = \CC \pow{t_1, \dots, t_{24}} / (t_1 t_2, t_1 t_3, t_4 t_5, t_4 t_6)
\]
be the miniversal deformation base space of $X$.
The group $G$ acts on $A$ and the invariant subring $A^G$ is the completion of $\modspace{3}{12}$ at the closed point corresponding to $X$.
Recall that $G \simeq (T_N)\rtimes \Aut(P)\simeq (\CC^*)^3 \rtimes (D_6 \rtimes C_2)$, we understand the action of $G$ on $A$ by looking at the action of the factors in this semidirect product decomposition. 
More precisely:
    \begin{enumerate}[label=$\bullet$]
\item $T_N = (\CC^*)^3$ acts diagonally with weights given by the characters of the $T_N$-representation $\TT^1_X$. More precisely:  $t_1, t_2, t_3$ have degree $(0,0,1) \in M$, $t_4, t_5, t_6$ have degree $(0,0,-1) \in M$, and the degrees of $t_7, \dots, t_{24}$ are the $18$ elements in $M$ depicted in Figure~\ref{fig:esagono_solo}.
\item
$D_6$ fixes $t_1,\dots, t_6$ and permutes $t_7, \dots, t_{24}$, and 
\item $C_2$ fixes $t_7, \dots, t_{24}$ and swaps $t_1$ and $t_4$, $t_2$ and $t_5$, and $t_3$ and $t_6$.
\end{enumerate}

Since the $G$-action on $A$ is linear in $t_1, \dots, t_{24}$ we can work with
the $\CC$-algebra
\[
B = \CC [t_1, \dots, t_{24}]/(t_1 t_2, t_1 t_3, t_4 t_5, t_4 t_6)
\]
together with the $G$-action, rather than having to consider power series.
The completion of $B$ at the origin is $A$ and that of $B^G$ at the origin will be $A^G$.

Consider the $\CC$-algebras
\begin{align*}
    R &= \CC [t_1, \dots, t_{6}]/(t_1 t_2, t_1 t_3, t_4 t_5, t_4 t_6), \mbox{ and } \\
    S &= \CC[t_7, \dots, t_{24}].
\end{align*}
Consider the subtori $\CC^*$ and $(\CC^*)^2$ of $T_N$ with cocharacter lattices $\{(0,0)\} \times \ZZ \subseteq N$ and $\ZZ^2 \times \{0\} \subseteq N$, respectively, then:
 \[B = R \otimes_\CC S\] and the actions of $(\CC^*)$ and of $C_2$ are non-trivial only on $R$, while the actions of $(\CC^*)^2$ and of $D_6$ are non-trivial only on $S$.
Therefore we have:
\[
B^G = (R^{\CC^*})^{C_2} \otimes_\CC (S^{(\CC^*)^2})^{D_6}.
\]
Let us analyse the two rings of invariants $(R^{\CC^*})^{C_2}$ and $(S^{(\CC^*)^2})^{D_6}$.

The ring $S^{(\CC^*)^2}$ defines a toric affine variety of dimension $16$. Taking the finite quotient under the action of $D_6$, shows that $(S^{(\CC^*)^2})^{D_6}$ is a normal domain of dimension $16$.

Let us consider the action of $\CC^*$ on the polynomial ring $\CC[t_1, \dots, t_6]$: the invariant subring $\CC [t_1, \dots, t_6]^{\CC^*}$ is generated by
\[
\begin{matrix}
y_0 = t_1 t_4 & y_3 = t_2 t_4 & y_6 = t_3 t_4 \\
y_1 = t_1 t_5 & y_4 = t_2 t_5 & y_7 = t_3 t_5 \\
y_2 = t_1 t_6 & y_5 = t_2 t_6 & y_8 = t_3 t_6
\end{matrix}
\]
as a $\CC$-algebra.
Therefore $\CC [t_1, \dots, t_6]^{\CC^*}$ is the quotient of $\CC [y_0, \dots, y_8]$ with respect to the ideal generated by the $2 \times 2$-minors of the $3 \times 3$-matrix containing $y_0, \dots, y_8$ as above. In other words, $\CC [t_1, \dots, t_6]^{\CC^*}$ is the homogeneous coordinate ring of the Segre embedding $\PP^2 \times \PP^2 \into \PP^8$.
Since $\CC^*$ is reductive, the surjection $\CC[t_1, \dots, t_6] \onto R$ induces a surjection $\CC[t_1, \dots, t_6]^{\CC^*} \onto R^{\CC^*}$. Consequently we have a surjection $\CC[y_0, \dots, y_8] \onto R^{\CC^*}$; its kernel is the ideal generated by
\begin{gather*}
    y_0 y_1, \
    y_0 y_2, \
    y_0 y_3, \
    y_0 y_4, \
    y_0 y_5, \
    y_0 y_6, \
    y_0 y_7, \
    y_0 y_8, \\
    y_1 y_3, \
    y_1 y_4, \ 
    y_1 y_5, \
    y_1 y_6, \
    y_1 y_7, \
    y_1 y_8, \\
    y_2 y_3, \
    y_2 y_4, \ 
    y_2 y_5, \
    y_2 y_6, \
    y_2 y_7, \
    y_2 y_8, \\
    y_3 y_4, \
    y_3 y_5, \
    y_3 y_7, \
    y_3 y_8, \
    y_4 y_6, \\
    y_4 y_8 - y_5 y_7, \
    y_5 y_6, \ 
    y_6 y_7, \
    y_6 y_8.
\end{gather*}
The action of $C_2$ on $\CC[y_0, \dots, y_8]$ is as follows: $y_0$, $y_4$, $y_8$ are kept fixed, $y_1$ is swapped with $y_3$, $y_2$ is swapped with $y_6$, and $y_5$ is swapped with $y_7$.
Therefore the invariant subring $\CC[y_0, \dots, y_8]^{C_2}$ is the polynomial ring $\CC[z_1, \dots, z_9]$ where
\begin{gather*}
    z_1 = y_0 \quad z_2 = y_4 \quad z_3 = y_8, \\
    z_4 = y_1 + y_3 \quad z_5 = y_1 y_3, \\
    z_6 = y_2 + y_6 \quad z_7 = y_2 y_6, \\
    z_8 = y_5 + y_7 \quad z_9 = y_5 y_7.
\end{gather*}
Since $C_2$ is a finite group, the surjection $\CC[y_0, \dots, y_8] \onto R^{\CC^*}$ induces a surjection
$\CC[z_1, \dots, z_9] \onto (R^{\CC^*})^{C_2}$; one checks that the kernel of this surjection is the ideal generated by
\begin{gather*}
    z_5, \
    z_7, \
    z_1 z_2, \
    z_1 z_3, \
    z_1 z_4, \
    z_1 z_6, \
    z_1 z_8, \
    z_1 z_9,
    \\
    z_2 z_4, \ 
    z_2 z_6, \
    z_3 z_4, \
    z_3 z_6, \ 
    z_4 z_8, \
    z_4 z_9, \
    z_6 z_8, \
    z_6 z_9,  \\
    z_2 z_3 - z_9. 
\end{gather*}
One checks that this ideal is the intersection of the following $3$ prime ideals
\begin{gather*}
    (z_1,
    z_2 z_3 - z_9,
    z_4,
    z_5,
    z_6,
    z_7),
    \\
    (  z_1,
    z_2,
    z_3,
    z_5,
    z_7,
    z_8,
    z_9),
    \\
    (    z_2,
    z_3,
    z_4,
    z_5,
    z_6,
    z_7,
    z_8,
    z_9).
\end{gather*}
This implies that $\Spec  (R^{\CC^*})^{C_2}$ is reduced and has $3$ irreducible components, which are all smooth and have dimensions $3$, $2$, and $1$.

Therefore $\Spec A^G$ is reduced and has $3$ irreducible components, all of which are normal and have dimensions $19$, $18$, and $17$.
We have proved: 

\begin{theorem}\label{Kmod-ex} Let $X$ be the Fano $3$-fold in Theorem~\ref{thm:example}. The K-moduli stack $\stack{3}{12}$ is reduced at $[X]$ and its completion at $[X]$ has $4$ branches. The K-moduli space $\modspace{3}{12}$  is reduced at $[X]$ and its completion at $[X]$ has $3$ irreducible components. 
\end{theorem}

\section{Non-reduced K-moduli and other pathologies}

\subsection{Non-reduced K-moduli}
\label{sec:non_reduced_K_moduli}
In this section, we prove Theorem~\ref{non-red}. 
Let $P$ be the convex hull of the following points in $N = \ZZ^3$ (see Figure~\ref{fig:polytope_non_reduced}):
\begin{align*}
    a_0 &= \begin{pmatrix}
        1 \\ 0 \\ 1
    \end{pmatrix}
    , &
    d_0 &= \begin{pmatrix}
        1 \\ 1 \\ 1
    \end{pmatrix}
    , &
    c_1 &= \begin{pmatrix}
        0 \\ 1 \\ 1
    \end{pmatrix}
    , & 
    b_1 &= \begin{pmatrix}
        -1 \\ -1 \\ 1
    \end{pmatrix}
    \\
    a_1 &= \begin{pmatrix}
        -1 \\ 0 \\ -1
    \end{pmatrix}
    , &
    d_1 &= \begin{pmatrix}
        -1 \\ -1 \\ -1
    \end{pmatrix}
    , &
    c_0 &= \begin{pmatrix}
        0 \\ -1 \\ -1
    \end{pmatrix}
    , & 
    b_0 &= \begin{pmatrix}
        1 \\ 1 \\ -1
    \end{pmatrix}.
\end{align*}

\begin{figure}
    \centering
    \includegraphics[width=5cm,scale = 0.75]{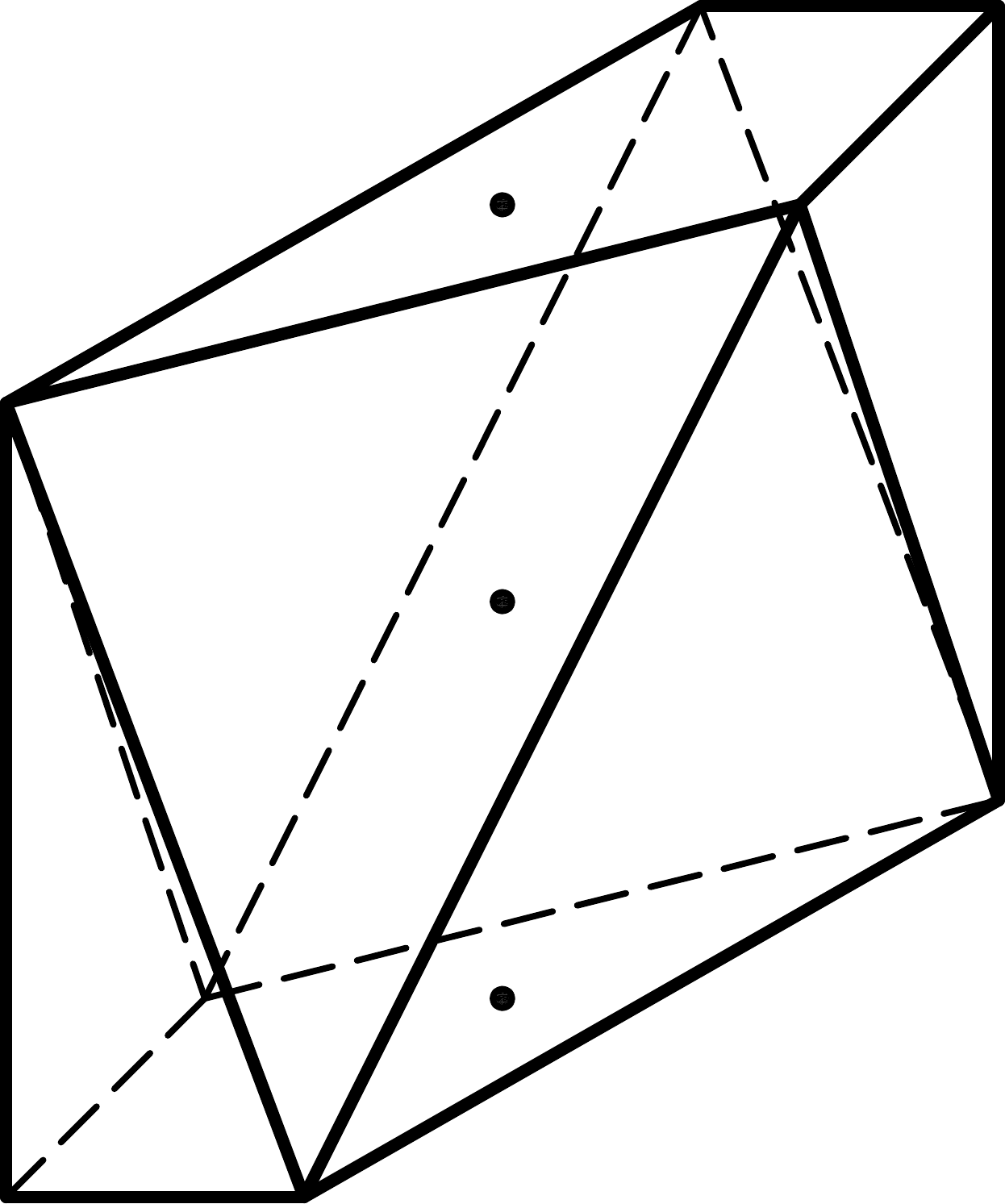}
    \caption{The polytope $P$ in \S\ref{sec:non_reduced_K_moduli}}
    \label{fig:polytope_non_reduced}
\end{figure}

The toric variety $X$ associated to the face fan of $P$ is a Fano $3$-fold with canonical singularities.
Since $P$ is centrally symmetric, its polar $P^\circ$ is also centrally symmetric and hence $X$ is K-polystable by Theorem~\ref{toricKps}. The anticanonical degree of $X$ is the normalised volume of $P^\circ$, so that $(-K_X)^3= \frac{44}{3}$. 

The singular locus of $X$ has $8$ irreducible connected components; these are: 
\begin{enumerate}
    \item For $i \in \{0,1\}$ the facet $\conv{a_i, d_i, c_{1-i}, b_{1-i}}$ gives a Gorenstein isolated singularity isomorphic to the affine cone over the anticanonical embedding of $\FF_1$ into $\PP^8$, where $\FF_1 = \PP(\cO_{\PP^1} \oplus \cO_{\PP^1} (1))$ is the $1$st Hirzebruch surface. By \cite[9.1.iv]{altmann_versal} the hull of the deformation functor of this singularity is $\CC [t] / (t^2)$.
    
    \item The $4$ facets $\conv{a_i, b_i, c_i}$ and $\conv{a_i, c_i, b_{1-i}}$, for $i \in \{0,1\}$,  give the non-Gorenstein isolated singularity $\frac{1}{3}(1,1,2)$, whose canonical cover is $\AA^3$. Therefore these singularities do not contribute to $\TTqG{1}{X}$ and $\TTqG{2}{X}$.
    
    \item For $i \in \{0,1\}$ the closure of the torus orbit corresponding to the edge $\conv{b_i, d_i}$ is a curve $C_i \simeq \PP^1$ along which $X$ has transverse $A_1$ singularities. By \cite{petracci_aft}, in a neighbourhood of $C_i$ the sheaf $\cT^1_X$ coincides with $\cTqG{1}{X}$ and with $\cO_{C_i}(-2)$. Therefore this sheaf has trivial $\rH^0$ and $1$-dimensional $\rH^1$.
\end{enumerate}

Let $U_0$ and $U_1$ be affine neighbourhoods of the isolated Gorenstein singularities of $X$; then:

\begin{align*}
    \TTqG{1}{X} &= \TT^1_{U_0} \oplus \TT^1_{U_1} = \CC^2, \\
    \TTqG{2}{X} &= \TT^2_{U_0} \oplus \TT^2_{U_1} \oplus \rH^1(C_0, \cO_{C_0}(-2)) \oplus \rH^1(C_1, \cO_{C_1}(-2)) = \CC^4.
\end{align*}
The degrees in $M$ of $\TTqG{1}{X}$ are $(0,0,1)$ and $(0,0,-1)$.
The degrees in $M$ of $\TTqG{2}{X}$ are $(0,0,2)$, $(0,0,-2)$, $(1,1,0)$, $(-1,-1,0)$.
This implies that the base of the miniversal $\QQ$-Gorenstein deformation of $X$ is
\[
A = \CC [t_0, t_1]/(t_0^2, t_1^2)
\]
where $t_0$ (resp.\ $t_1$) has degree $(0,0,1)$ (resp.\ $(0,0,-1)$) in $M$.

The automorphism group $\Aut(P)$ of $P$ is generated by the reflection about the origin $- \mathrm{id}_N$ and by the reflection across the plane containing $b_0$, $b_1$, $d_0$, $d_1$.
The group $\Aut(P)$ is isomorphic to $C_2 \times C_2$.
Since no facet of $P^\circ$ has interior lattice points, by Proposition~\ref{prop:automorphism_toric_fano} $G = \Aut(X)$ is the semidirect product of the torus $T_N$ with $\Aut(P)$.

Let us now determine the subring of invariants $A^G \subseteq A$.
The action of the torus $T_N = (\CC^*)^3$ on $A$ is described as follows: the degree of $t_0$ (resp.\ $t_1$) is the character $(0,0,1)$ (resp.\ $(0,0,-1)$) in $M$.
Therefore $A^{(\CC^*)^3} = \CC [t]/(t^2)$, where $t = t_0 t_1$.
The group $\Aut(P) = C_2 \times C_2$ acts trivially on $A^{(\CC^*)^3}$, hence
\begin{equation*}
    A^G = \CC[t] / (t^2).
\end{equation*}
The spectrum of this $0$-dimensional ring is the connected component of $\modspace{3}{44/3}$ containing the point corresponding to $X$. This proves Theorem~\ref{non-red}.

\subsection{ K-moduli of products}

\begin{proposition} \label{prop:product_fano}
   Let $X$ and $Y$ be log terminal Fano varieties. If $Y$ is Gorenstein, then the natural map
   \begin{equation} \label{eq:product_map}
   \DefqG X \times \Def Y \longrightarrow \DefqG X \times Y
   \end{equation}
   is formally smooth and induces an isomorphism on tangent spaces.
\end{proposition}

\begin{proof}
    Let $\epsi \colon \frakX \to X$ be the canonical cover stack.
    Then the product map $\epsi \times \id_Y$ is the canonical cover stack of $X \times Y$.
    Let $p \colon \frakX \times Y \to \frakX$ and $q \colon \frakX \times Y \to Y$ be the projections.
    
        By Kawamata--Viehweg vanishing, $\rH^i(\cO_Y) = 0$ for every $i>0$.
    This implies that $\rR p_* \cO_{\frakX \times Y} = \cO_\frakX$.
    For every $i \geq 0$ we have
    \begin{align*}
        \Ext^i_{\frakX \times Y} (p^* \LL_\frakX, \cO_{\frakX \times Y}) &= \Ext^i_\frakX (\LL_\frakX , \rR p_* \cO_{\frakX \times Y}) \\
        &= \Ext^i_\frakX (\LL_\frakX , \cO_{\frakX}) \\
        &= \TTqG{i}{X}.
    \end{align*}
    
    Since $\epsi$ is cohomologically affine and $\epsi_* \cO_\frakX = \cO_X$, by Kawamata--Viehweg vanishing,  we have $\rH^0(\cO_\frakX) = \CC$ and 
    $\rH^i(\cO_\frakX) = 0$ for every $i >0$. This implies that $\rR q_* \cO_{\frakX \times Y} = \cO_Y$. 
    For every $i \geq 0$ we have
    \begin{align*}
    \Ext^i_{\frakX \times Y} (q^* \LL_Y, \cO_{\frakX \times Y}) &= \Ext^i_Y (\LL_Y , \rR q_* \cO_{\frakX \times Y}) \\
    &= \Ext^i_Y (\LL_Y , \cO_Y) \\
    &= \TT^{i}_{Y}. 
\end{align*}

    The decomposition of the cotangent complex
    \[
    \LL_{\frakX \times Y} = p^* \LL_\frakX \oplus q^* \LL_Y
    \]
    implies
    \begin{align*}
        \TTqG{i}{X \times Y} &= \Ext^i_{\frakX \times Y}(\LL_{\frakX \times Y}, \cO_{\frakX \times Y}) \\
        &= \Ext^i_{\frakX \times Y} (p^* \LL_\frakX, \cO_{\frakX \times Y}) \oplus \Ext^i_{\frakX \times Y} (q^* \LL_Y, \cO_{\frakX \times Y}) \\
        &= \TTqG{i}{X} \oplus \TT^{i}_{Y}
        \end{align*}
    for every $i \geq 0$.
   Therefore, the map \eqref{eq:product_map} induces a bijection on tangent spaces and an injection on obstruction spaces.
    \end{proof}
This result shows that we can generalise examples of pathological K-moduli spaces to higher dimension. For example, an immediate corollary of Theorem~\ref{thm:example} is obtained by considering the products $X \times \PP^{n-3}$, where $X$ denotes the Fano variety in that Theorem. We obtain: 

\begin{theorem} Let $X$ be the K-polystable toric Fano $3$-fold $X$ considered in Theorem~\ref{thm:example}.
If $n \geq 4$, $V = 2n(n-1)(n-2)^{n-2}$ ,
then $\stack{n}{V}$ and $\modspace{n}{V}$ are not smooth at the point corresponding to $X \times \PP^{n-3}$.
\end{theorem}

\begin{proof}
Set $X' = X \times \PP^{n-3}$, which is a K-polystable toric Fano $n$-fold with anticanonical degree
    \[
    V = {n \choose 3} \times 12 \times (n-2)^{n-3}.
    \]
    Let $A$ be the base of the miniversal deformation of $X$ and let $G$ be the automorphism group of $X$.
    By Proposition~\ref{prop:product_fano} the base of the miniversal deformation of $X'$ is $A$.
    One can check that the automorphism group of $X'$ is $G \times \Aut(\PP^{n-3})$.
    This implies that the local structure of $\stack{n}{V}$ is $[\Spec A \ / \ G] \times \mathrm{B} \, \mathrm{PGL}_{n-2}$ and the local structure of $\modspace{n}{V}$ is $\Spec A^G$.
\end{proof}

\begin{remark}
Similar conclusions can be reached when considering products $X\times Z$ for any smooth K-polystable Fano $Z$. \end{remark}

\begin{remark}The product of the toric Fano variety of Theorem~\ref{thm:example} with the toric Fano variety of 
 \S\ref{sec:non_reduced_K_moduli} has dimension $6$ and anticanonical degree 
 \[ {6\choose3}\times 12\times \frac{44}{3}= 3520.\] Using Proposition~\ref{prop:product_fano}, one can show that there is a point of $\stack{6}{3520}$, at which the completion is non-reduced and reducible. 
\end{remark}

\section{K-stability results from explicit smoothings of toric Fano $3$-folds}
\label{sec:Ksstab}

We first prove the following general result. 

\begin{proposition}\label{K-ps}
Let $P$ be a Fano polytope in a lattice $N$ and let $X$ be the toric Fano variety associated to the face fan of $P$. Assume that $P$ is centrally symmetric, i.e.\ $P = -P$, and that no facet of the polar of $P$ has interior lattice points.

If $X$ has unobstructed $\QQ$-Gorenstein deformations, then the general $\QQ$-Gorenstein deformation of $X$ is K-polystable.
\end{proposition}

\begin{proof}
Since $P$ is centrally symmetric, the polar $P^\circ$ is centrally symmetric and hence the barycentre of $P^\circ$ is the origin.
Therefore $X$ is K-polystable by Theorem~\ref{toricKps} and determines a closed point $[X]$ in $\stack{n}{V}$ and in $\modspace{n}{V}$ for suitable $n$ and $V$.

Let $\cX_t$ denote a general fibre in the miniversal $\QQ$-Gorenstein deformation of $X$.
By openness of K-semistability \cite{BLX, xu_minimizing}, we immediately deduce that $\cX_t$ is K-semistable.
However, the local description of K-moduli allows us to prove that $\cX_t$ is actually K-polystable as follows.

Consider the dual lattice $M = \Hom_\ZZ(N,\ZZ)$ and  the torus $T_N = \Spec \CC[M]$ which acts on $X$. Therefore $\TTqG{1}{X}$ is a linear $T_N$-representation, in particular $\TTqG{1}{X}$ is the direct sum of $1$-dimensional linear representations of $T_N$, i.e.\ characters of $T_N$.
Let $W$ denote the weight polytope of this linear action, i.e.\ $W$ is the convex hull in $M \otimes_\ZZ \RR$ of the points in $M$ corresponding to the characters of $T_N$ which appear in $\TTqG{1}{X}$.
Since $P$ is centrally symmetric, also $W$ is centrally symmetric; hence the origin lies in the relative interior of $W$.
By \cite[Exercise~5.21c]{szekelyhidi} the general point in the affine space $\TTqG{1}{X}$ is GIT-polystable with respect to the $T_N$-action.
By Proposition~\ref{prop:automorphism_toric_fano} the automorphism group of $X$ is a finite extension of $T_N$, therefore the general point in the affine space $\TTqG{1}{X}$ belongs to a closed orbit with respect to the $\Aut(X)$-action.

Since $\QQ$-Gorenstein deformations of $X$ are unobstructed, the germ $\DefqG X$ coincides with (a neighbourhood of the origin in) its tangent space $\TTqG{1}{X}$. By the Luna \'etale slice theorem for algebraic stacks \cite{luna_etale_slice_stacks} (see \S\ref{sec:KstabFa}) we deduce that $\cX_t$ gives rise to a closed point in $\stack{n}{V}$, so that $\cX_t$ is K-polystable.
\end{proof}

\subsection{K-polystability of \morimukai{4}{3}}
\label{sec:MM4-3}

Let $P$ be the lattice polytope in $N = \ZZ^3$ with vertices:
\begin{equation*}
    \pm \begin{pmatrix}
        1 \\ 0 \\ 0 
    \end{pmatrix}\!, \
    \pm \begin{pmatrix}
        0 \\ 1 \\ 0 
    \end{pmatrix}\!, \
    \pm \begin{pmatrix}
        0 \\ 0 \\ 1
    \end{pmatrix}\!, \
    \pm \begin{pmatrix}
        1 \\ 1 \\ 0
    \end{pmatrix}\!, \
    \pm \begin{pmatrix}
        1 \\ 0 \\ 1
    \end{pmatrix}\!.
\end{equation*}
Then $P$ has exactly $12$ facets: $4$ of them are standard squares (i.e.\ quadrilaterals affine-equivalent to the convex hull of $0$, $e_1$, $e_2$, $e_1+ e_2$ in $\ZZ^2$) and $8$ of them are standard triangles (i.e.\ triangles affine-equivalent to the convex hull of $0$, $e_1$, $e_2$ in $\ZZ^2$). One can see that $P$ is a reflexive polytope.

Let $X$ be the Gorenstein toric Fano $3$-fold associated to the face fan of $P$. 
Since $P$ is centrally symmetric, %(i.e.\ $P = -P$)
$P^\circ$ also is centrally symmetric and the barycentre of $P^\circ$ is zero so that $X$ is K-polystable by Theorem~\ref{toricKps}.

From the description of the facets of $P$, we conclude that $\Sing X$ consists of $4$ ordinary double points. Therefore, $X$ is unobstructed and smoothable by \cite{namikawa_fano}. We now prove that $X$ deforms to a member of the family \morimukai{4}{3}. This will imply that the general member of \morimukai{4}{3} is K-semistable by Theorem~\ref{openness}.

As the normalised volume of $P^\circ$ is $28$, $(-K_X)^3 = 28$. By Proposition~\ref{prop:jordan}, $b_2(X) = 4$. Since $X$ has ordinary double points, each Milnor fibre $M$ associated to the local smoothing of a singular point of $X$ induced by a smoothing of $X$ is homotopy equivalent to $S^3$, so $b_1(M)=b_2(M)=0$. Via the long exact sequence \eqref{les}, this implies that the second Betti number is constant in the family. In other words, the Picard rank of the general fibre of the smoothing of $X$ is $4$.
By inspection of the list of smooth Fano $3$-folds, one sees that \morimukai{4}{3} is the only family of smooth Fano $3$-folds with Picard rank $4$ and anticanonical degree $28$. This implies that $X$ deforms to \morimukai{4}{3}.

By Proposition~\ref{K-ps}, the general member of \morimukai{4}{3} is K-polystable; since the automorphism group of the general member of \morimukai{4}{3} is infinite \cite{Prz_Che_Shr}, it is not K-stable by \cite[Corollary 1.3]{blum_xu_uniqueness}.

\subsection{K-stability of \morimukai{2}{10}} \label{sec:MM2-10} In this section, we construct a K-polystable Gorenstein toric Fano variety and show that it has a deformation to \morimukai{2}{10}.

Let $P$ be the convex hull of the following points in the lattice $N = \ZZ^3$:
\begin{equation*}
\begin{pmatrix}
    0 \\ 0 \\  1
\end{pmatrix},
\begin{pmatrix}
    1 \\ 0 \\  1
\end{pmatrix},
\begin{pmatrix}
    1 \\ 1 \\  1
\end{pmatrix},
\begin{pmatrix}
    0 \\ 1 \\  1
\end{pmatrix},
\pm
\begin{pmatrix}
     1 \\ 1 \\  0
\end{pmatrix},
\pm
\begin{pmatrix}
     1 \\ - 1 \\  0
\end{pmatrix},
\begin{pmatrix}
    0 \\ 0 \\  -1
\end{pmatrix},
\begin{pmatrix}
    -1 \\ 0 \\  -1
\end{pmatrix},
\begin{pmatrix}
    -1 \\ -1 \\  -1
\end{pmatrix},
\begin{pmatrix}
    0 \\ -1 \\  -1
\end{pmatrix}.
\end{equation*}
These $12$ points are exactly the vertices of $P$. The polytope $P$ has $20$ edges and $10$ facets, which are all quadrilateral. One can see that $P$ is a bifrustum, i.e.\ the union of two frusta joined at their common bases.
The polytope $P$ is depicted in Figure~\ref{fig:due_tronchi}.

\begin{figure}
    \centering
    \includegraphics[width=5cm, scale = 0.75]{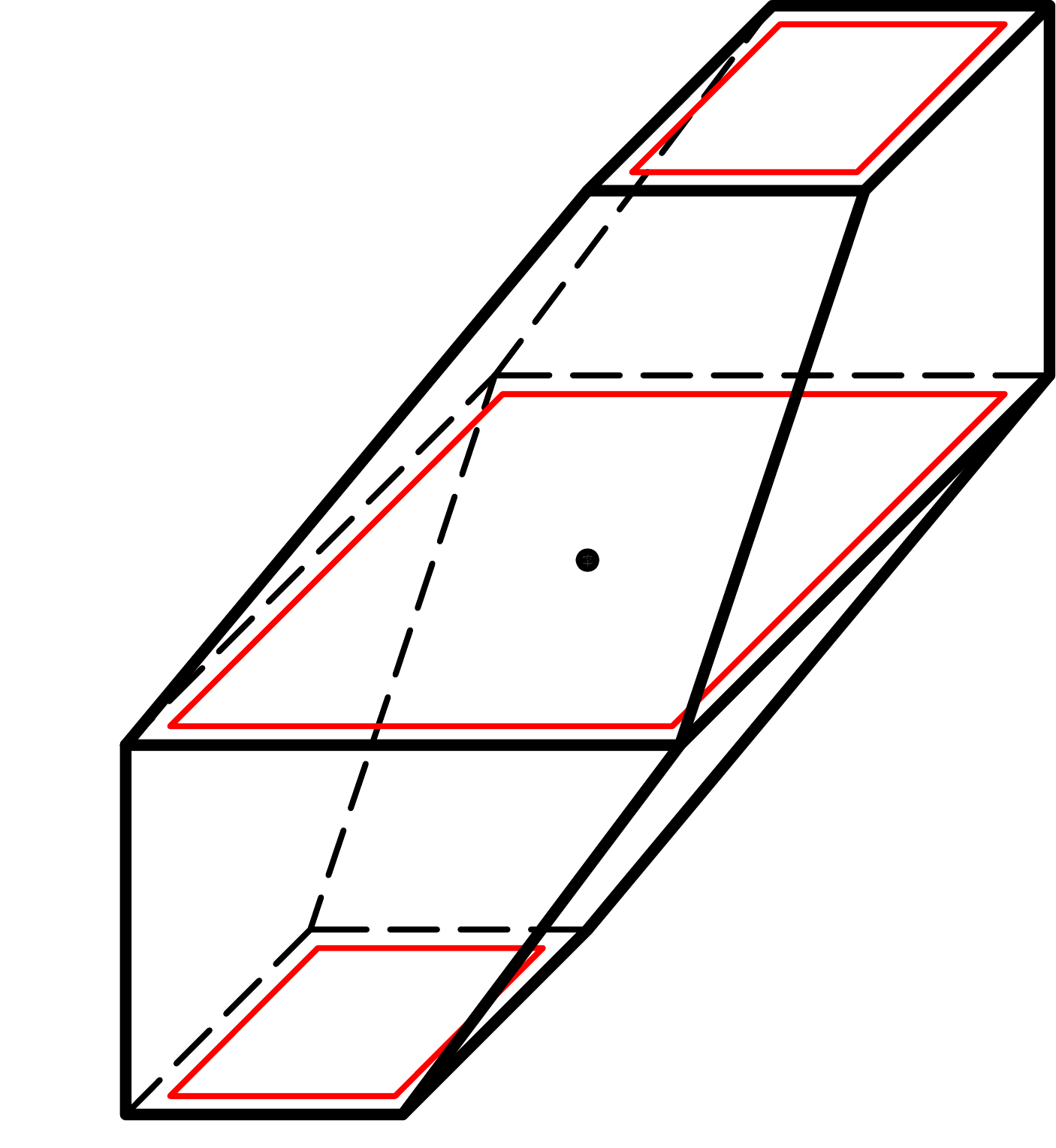}
    \caption{The polytope $P$ considered in \S\ref{sec:MM2-10}}
    \label{fig:due_tronchi}
\end{figure}

The toric variety $X$ associated to the face fan of $P$ is Fano with Gorenstein canonical hypersurface singularities.

We want to prove that $X$ admits a closed embedding into a smooth toric variety as a complete intersection.
We follow the Laurent inversion method developed by Coates--Kasprzyk--Prince \cite{laurent_inversion, from_cracked_to_fano, cracked_fano_toric_ci}.
Consider the decomposition
\[
N = \overline{N} \oplus N_U
\]
where $\overline{N} = \ZZ e_1 \oplus \ZZ e_2$ and $N_U = \ZZ e_3$.
Let $\overline{M}$ be the dual lattice of $\overline{N}$.
Let $Z$ be the $T_{\overline{M}}$-toric variety associated to the complete fan in the lattice $\overline{M}$ with rays generated by $e_1^*, e_2^*, -e_1^*, -e_2^* \in \overline{M}$.
It is clear that $Z$ is isomorphic to $\PP^1 \times \PP^1$.
Let $\Div_{T_{\overline{M}}}(Z)$ be the rank-$4$ lattice made up of the torus invariant divisors on $Z$: a basis of $\Div_{T_{\overline{M}}}(Z)$ is given by $E_1, E_2, E_3, E_4$, the torus invariant prime divisors on $Z$ associated to the rays generated by $e_1^*, e_2^*, -e_1^*, -e_2^*$ respectively.
The divisor sequence \cite[Theorem~4.1.3]{cox_toric_varieties} 
\begin{equation*}
0 \longrightarrow \overline{N} \overset{\rho^\star}\longrightarrow
\Div_{T_{\overline{M}}}(Z)
\longrightarrow \Pic(Z)
\longrightarrow 0
\end{equation*}
of $Z$ becomes
\begin{equation*}
0 \longrightarrow \overline{N} = \ZZ e_1 \oplus \ZZ e_2
\xrightarrow{\begin{pmatrix}
1 & 0 \\
0 & 1 \\
-1 & 0 \\
0 & -1
    \end{pmatrix}}
\ZZ^4
\xrightarrow{\begin{pmatrix}
    1 & 0 & 1 & 0 \\
    0 & 1 & 0 & 1
    \end{pmatrix}}
\ZZ^2 \longrightarrow 0.
\end{equation*}
We consider the following ample torus invariant divisors on $Z$
\begin{align*}
D_0 &= E_3 + E_4 \in \Div_{T_{\overline{M}}}(Z), \\
D_1 &= E_1 + E_2 \in \Div_{T_{\overline{M}}}(Z), \\
D_x &= E_1 + E_2 + E_3 + E_4 \in \Div_{T_{\overline{M}}}(Z)
\end{align*}
and their corresponding moment polytopes in $\overline{N}$
\begin{align*}
P_{D_0} &= \conv{0, e_1, e_2, e_1 + e_2}, \\
P_{D_1} &= \conv{0, -e_1, -e_2, -e_1 - e_2}, \\
P_{D_x} &=   \conv{e_1+e_2, e_1- e_2, -e_1+e_2, -e_1-e_2}.
\end{align*}
Now consider the following elements of the lattice $N_U = \ZZ e_3$:
\begin{equation*}
\chi_0 = e_3, \qquad \chi_1 = -e_3, \qquad \chi_x = 0.
\end{equation*}
The polytopes $ P_{D_0} + \chi_0$, $ P_{D_1} + \chi_1$, $ P_{D_x} + \chi_x$ in $N = \overline{N} \oplus N_U$ are depicted in red in Figure~\ref{fig:due_tronchi}. Clearly the polytope $P$ is the convex hull of these three polytopes.
By \cite[Definition~3.1]{laurent_inversion} the set
\begin{equation*}
S = \{ (D_0, \chi_0), (D_1, \chi_1), (D_x, \chi_x)  \}
\end{equation*}
is a scaffolding on the Fano polytope $P$.

Set $\widetilde{N} := \Div_{T_{\overline{M}}}(Z) \oplus N_U$. 
Let $\widetilde{M}$ be the dual lattice of $\widetilde{N}$.
Let $\langle \cdot, \cdot \rangle \colon \widetilde{M} \times \widetilde{N} \to \ZZ$ be the duality pairing.
Following \cite[Definition~A.1]{laurent_inversion} we consider the polytope $Q_S \subseteq  \widetilde{M}_\RR$ defined by the following inequalities:
\begin{align*}
\langle \ \cdot \ , -D_0 + \chi_0 \rangle &\geq -1, \\
\langle \ \cdot \ , -D_1 + \chi_1 \rangle &\geq -1, \\
\langle \ \cdot \ , -D_x + \chi_x \rangle &\geq -1, \\
\langle \ \cdot \ , E_1 \rangle &\geq 0, \\
\langle \ \cdot \ , E_2 \rangle &\geq 0, \\
\langle \ \cdot \ , E_3 \rangle &\geq 0, \\
\langle \ \cdot \ , E_4 \rangle &\geq 0.
\end{align*}
Let $\Sigma_S$ be the normal fan of $Q_S$, then $\Sigma_S$ is a smooth complete fan in $\widetilde{N} = \Div_{T_{\overline{M}}}(Z) \oplus N_U$ and its rays are generated by the vectors
\begin{align*}
    s_0 &= -D_0 + \chi_0 = -E_3 - E_4 + e_3, \\
    s_1 &= -D_1 + \chi_1 = -E_1 - E_2 - e_3, \\
    x &= -D_x + \chi_x = -E_1 - E_2 - E_3 -E_4, \\
    x_2 &= E_1, \\
    x_3 &= E_2, \\
    x_4 &= E_3, \\
    x_5 &= E_4.
    \end{align*}

Let $Y$ be the $T_{\widetilde{N}}$-toric variety associated to the fan $\Sigma_S$. Thus $Y$ is a smooth $5$-fold with Cox coordinates $s_0$, $s_1$, $x$, $x_2$, $x_3$, $x_4$, $x_5$.
In the basis of $\widetilde{N}$ given by $E_1$, $E_2$, $E_3$, $E_4$, $e_3$,
the ray map $\ZZ^7 \to \widetilde{N}$ is given by the matrix
 \begin{equation*}
 \begin{pmatrix}
 0 & -1 & -1 & 1 & 0 & 0 & 0 \\
 0 & -1 & -1 & 0 & 1 & 0 & 0 \\
 -1 & 0 & -1 & 0 & 0 & 1 & 0 \\
 -1 & 0 & -1 & 0 & 0 & 0 & 1 \\
 1 & -1 & 0  & 0 & 0 & 0 & 0 
 \end{pmatrix}.
 \end{equation*}
After computing the kernel of this matrix, one finds that the divisor map $\ZZ^7 \to \Pic(Y) \simeq \ZZ^2$ is given by the following matrix.
\begin{equation*}
\begin{array}{ccccccc|c}
s_0 & s_1 & x & x_2 & x_3 & x_4 & x_5 \\
\hline
1 & 1 & -1 & 0 & 0 &  0 & 0 & L_1 \\
0 & 0 & 1 & 1 & 1 & 1 & 1 & L_2
\end{array}
\end{equation*}
This matrix gives also the weights of an action of $\GG_\rmm^2$ on $\AA^7$: $Y$ is the GIT quotient of this action with respect to the stability condition given by the irrelevant ideal $(s_0, s_1)\cdot (x, x_2, x_3, x_4, x_5)$. Here $L_1$ and $L_2$ denote the elements of the chosen basis of $\Pic(Y)$.

The morphism $Y \to \PP^1$ given by $[s_0 : s_1 : x : x_2 : x_3 : x_4 : x_5] \mapsto [s_0 : s_1]$ shows that $Y$ is isomorphic to $\PP(\cO_{\PP^1}(-1) \oplus \cO_{\PP^1}^{\oplus 4})$.
The morphism $Y \to \PP^5$ given by
$[s_0 : s_1 : x : x_2 : x_3 : x_4 : x_5] \mapsto [s_0 x : s_1 x : x_2 : x_3 : x_4 : x_5]$
shows that $Y$ is isomorphic to the blowup of $\PP^5$ with centre the $3$-plane $(x_0 = x_1 = 0)$.

We now consider the injective linear map
\begin{equation*}
\theta := \rho^\star \oplus \mathrm{id}_{N_U} \colon
N = \overline{N} \oplus N_U \longrightarrow
\widetilde{N} = \Div_{T_{\overline{M}}}(Z) \oplus N_U.
\end{equation*}
By \cite[Theorem~5.5]{laurent_inversion} $\theta$ induces a toric morphism $X \to Y$ which is a closed embedding.
We want to understand the ideal of this closed embedding in the Cox ring of $Y$ using the map $\theta$.

We follow \cite[Remark~2.6]{from_cracked_to_fano}.
We see that $\theta(N) = (h_1)^\perp \cap (h_2)^\perp$, where $(h_1)^\perp$ (resp.\ $(h_2)^\perp$) is the hyperplane in $\widetilde{N}_\RR$ defined by the vanishing of $h_1 = E_1^* + E_3^* \in \widetilde{M}$ (resp.\ $h_2 = E_2^* + E_4^* \in \widetilde{M}$). From $\langle h_1, -D_0 + \chi_0 \rangle = \langle h_1, -D_1 + \chi_1 \rangle = -1$, $\langle h_1, -D_x + \chi_x \rangle = -2$, $\langle h_1,E_1 \rangle = \langle h_1,E_3 \rangle =1$, $\langle h_1,E_2 \rangle = \langle h_1,E_4 \rangle =0$, we obtain that the polynomial $x_2 x_4 - s_0 s_1 x^2$ lies in the homogeneous ideal of $X \into Y$. In an analogous way, from $h_2$ we obtain that the polynomial $x_3 x_5 - s_0 s_1 x^2$ lies in the homogeneous ideal of $X \into Y$. One can see that these two polynomials generate this ideal.

Therefore we have shown that $X$ is a complete intersection in $Y$: more specifically $X$ is the intersection of two particular divisors in the linear system $\vert 2 L_2 \vert$. This implies that $X$ deforms to the intersection of two general divisors in $\vert 2 L_2 \vert$, that is to a member of the family \morimukai{2}{10}.

\begin{lemma} \label{lem:big_nef}
$L_1 \vert_X$ is nef and $L_2 \vert_X$ is big and nef.
\end{lemma}

\begin{proof}
The smooth toric variety $Y$ has Picard rank $2$: the nef cone of $Y$ is generated by $L_1$ and $L_2$;
the effective cone of $Y$ is generated by $L_1$ and $-L_1 + L_2$.
%Recall that on a toric variety a Cartier divisor is nef if and only if it is base point free \cite[Theorem~6.3.12]{cox_toric_varieties}.
Since $L_1$ and $L_2$ are nef on $Y$, we immediately get that $L_1 \vert_X$ and $L_2 \vert_X$ are nef on $X$.

The section $s_0 \in \rH^0(Y, L_1)$ does not vanish identically on $X$, therefore $L_1 \vert_X$ is effective.
The section $x \in \rH^0(Y, -L_1 + L_2)$ does not vanish identically on $X$, therefore $(-L_1 + L_2) \vert_X$ is effective.
Since the Picard rank of $X$ is $2$, we have that $L_2 \vert_X = L_1 \vert_X + (-L_1 + L_2) \vert_X$ is in the interior of the effective cone, so it is big.
\end{proof}

\begin{lemma}
$X$ has unobstructed deformations.
\end{lemma}

\begin{proof}
By Lemma~\ref{lem:big_nef} and Demazure vanishing \cite[Theorem~9.2.3]{cox_toric_varieties}   $\cO_X$, $L_1 \vert_X$, $L_2 \vert_X$, $2L_2 \vert_X$ do not have higher cohomology.

It is clear that $-K_Y = L_1 + 5L_2$. By adjunction $K_X = -L_1 \vert_X - L_2 \vert_X$.
By Lemma~\ref{lem:big_nef}  and Kawamata--Viehweg vanishing $(-L_1 + L_2) \vert_X = K_X + 2 L_2 \vert_X$ does not have higher cohomology.

Consider the normal bundle $N_{X/Y}$ of $X$ in $Y$; one has $N_{X/Y} = \left( L_2^{\otimes 2} \vert_X \right)^{\oplus 2}$. We deduce that the higher cohomology of $N_{X/Y}$ vanishes. In particular we get the vanishing of $ \rH^1(N_{X/Y}) = \Ext^1(N_{X/Y}^\vee, \cO_X)$.

Consider the tangent bundle $T_Y$ of $Y$. By restricting the Euler sequence of $Y$ \cite[Theorem~8.1.6]{cox_toric_varieties} to $X$ we get the short exact sequence
\[
0 \to \cO_X^{\oplus 2} \to \left( L_1 \vert_X \right)^{\oplus 2} \oplus \left( L_1^\vee \otimes L_2 \right) \vert_X
\oplus \left( L_2 \vert_X \right)^{\oplus 4} \to T_Y \vert_X \to 0.
\]
We deduce that the higher cohomology of $T_Y \vert_X$ vanishes.
In particular we obtain the vanishing of $\rH^2(X, T_Y \vert_X) = \Ext^2(\Omega_Y \vert_X, \cO_X)$.

From the conormal sequence
\[
0 \to N_{X/Y}^\vee \to \Omega_Y \vert_X \to \Omega_X \to 0
\]
and the vanishing of $\Ext^1(N_{X/Y}^\vee, \cO_X)$ and of $\Ext^2(\Omega_Y \vert_X, \cO_X)$ we deduce the vanishing $\TT^2_X = \Ext^2(\Omega_X, \cO_X)$.
\end{proof}

By Proposition~\ref{K-ps} the general member of the family of \morimukai{2}{10} is K-polystable.
Since it has finite automorphism group \cite{Prz_Che_Shr},  it is K-stable.

We collect the results of this section in the next theorem: 
\begin{theorem}
There exist K-polystable toric Fano $3$-folds with smoothings to members of the deformation families \morimukai{4}{3} and \morimukai{2}{10}. The general member of the deformation family \morimukai{4}{3} is strictly K-polystable and the general member of the deformation family \morimukai{2}{10} is K-stable. 
\end{theorem}

\bibliography{Biblio_Kstable}
\end{document}